\documentclass[10pt]{article}
\usepackage[a4paper,top=4cm,bottom=4cm,left=2.5cm,right=2.5cm,bindingoffset=5mm]{geometry}
\usepackage{xcolor}

\usepackage[hyphens]{url}\usepackage[breaklinks=true,colorlinks,linkcolor={blue},citecolor={red},urlcolor={purple},]{hyperref}

\usepackage{todonotes}
\usepackage[font=small,format=default,indention=2em,labelsep = period]{caption} 
\usepackage{subfig}
\usepackage{enumitem}
\tikzset{>=latex}

\usepackage{amssymb,bm,amsmath,amsthm}
\usepackage{graphicx,algorithm,algorithmic,mathtools,mathrsfs}	

\newcommand{\R}{\mathbb R}

\newcommand{\norm}[1]{\left\lVert#1\right\rVert}
\newcommand{\disp}{\displaystyle}
\newcommand{\s}{\mathnormal{S}}
\newcommand{\I}{\mathnormal{I}}
\newcommand{\szero}{\s_{0}}
\newcommand{\izero}{\I_{0}}

\newcommand{\snc}{\mathcal{S}}
\newcommand{\Inc}{\mathcal{I}}
\newcommand{\szeronc}{\snc_{0}}
\newcommand{\izeronc}{\Inc_{0}}
\newcommand{\diffnc}{\mathcal{Z}}
\newcommand{\sommanc}{\mathcal{N}}
\newcommand{\nzeronc}{\sommanc_{0}}
\newcommand{\sis}{\textup{SIS}}
\newcommand\drl[1]{\mathcal{D}_{t}^{#1}}
\newcommand\dcap[1]{\mathscr{D}_{t}^{#1}}
\newcommand\rli[1]{\mathcal{I}^{#1}}

\newcommand{\sig}{\frac{\sigma-1}{\sigma}}
\newcommand\ek[1]{E^{\alpha}_{#1}}
\newcommand\ak[1]{A^{\alpha}_{#1}}

\newcommand{\dt}{\Delta t}
\newcommand{\nt}{N}
\newcommand{\tzero}{t_{0}}

\newcommand{\tn}{t_{n+1}}

\newcommand{\ralph}{r_{\alpha}}
\newcommand{\molt}{b}
\newcommand{\numuno}{M_{1}}
\newcommand{\numdue}{M_{2}}

\usepackage{titlesec}
\titleformat*{\section}{\Large\bfseries}
\titleformat*{\subsection}{\large\bfseries}
\let\oldparagraph=\paragraph
\renewcommand\paragraph[1]{\oldparagraph{#1.}}

\numberwithin{equation}{section}
\theoremstyle{definition}

\newtheorem{remark}{Remark}
\newtheorem{theorem}{Theorem}
\newtheorem{prop}{Proposition}

\numberwithin{remark}{section}
\numberwithin{theorem}{section}
\numberwithin{prop}{section}
\numberwithin{defn}{section}

\title{\LARGE\textbf{Fractional $\sis$ epidemic models}}
\author{\normalsize{Caterina Balzotti}\thanks{Department of Basic and Applied Sciences for Engineering, Sapienza University of Rome, Via A. Scarpa 16, Rome, Italy (\href{mailto:caterina.balzotti@sbai.uniroma1.it}{caterina.balzotti@sbai.uniroma1.it}, \href{paola.loreti@uniroma1.it }{paola.loreti@uniroma1.it })}
\and \normalsize{Mirko D'Ovidio}\thanks{Department of Basic and Applied Sciences for Engineering, Sapienza University of Rome, Via A. Scarpa 10, Rome, Italy (\href{mailto:mirko.dovidio@uniroma1.it}{mirko.dovidio@uniroma1.it}): corresponding author.
}
\and \normalsize{Paola Loreti}\footnotemark[1]}
\date{\vspace{-0.5cm}}

\allowdisplaybreaks

\begin{document}

\maketitle

\begin{abstract}
In this paper we consider the fractional $\sis$ epidemic model ($\alpha$-$\sis$ model) in the case of constant population size. We provide a representation of the explicit solution to the fractional model and we illustrate  the results by numerical schemes. A comparison with the limit case when the fractional order $\alpha \uparrow 1$ (the $\sis$ model) is also given. We analyse the effects of the fractional derivatives by comparing the $\sis$ and the $\alpha$-$\sis$ models. 
\end{abstract}

\begin{description}
\item[\textbf{Keywords.}] $\alpha$-$\sis$ model, $\sis$ model, epidemic models, fractional logistic equation.
\item[\textbf{Mathematics Subject Classification.}]  92D30, 78A70, 26A33.
\end{description}

\section{Introduction}
The study of mathematical models for epidemiology has a long history, dating back to the early 1900s with the theory developed by Kermack and McKendrick \cite{kermack1927RSL}. Such theory describes compartmental models, where the population is divided into groups depending on the state of individuals with respect to disease, distinguishing between groups. The dynamic of the disease is then described by a system of ordinary differential equations for each class of individuals.
The use of mathematical models for epidemiology is particularly useful to predict the progress of an infection and to take strategy to limit the spread of the disease.
In this work we focus on the $\alpha$-$\sis$ (susceptible - infectious - susceptible) epidemiological model. The $\sis$ model has a long history too \cite{hethcote1989AME}. It  describes  the spread of  human viruses such as influenza. The $\sis$ model with constant population is particularly appropriate to describe some bacterial agent diseases such as gonorrhea, meningitis and streptococcal sore throat. $\sis$ is a model without immunity, where the individual recovered from the infection comes back into the class of susceptibles.

\subsection{Statement of the problem}
We propose an $\alpha$-$\sis$ model with constant population size. 
The novelty  concerns the $\sis$ equations with the time fractional Caputo derivative in place of time standard derivative and their explicit solutions in terms of Euler's numbers and Euler's Gamma functions. 

Let us consider the Caputo fractional derivative introduced in \eqref{sec:fracCal} below. We provide an explicit representation of the solution to
\begin{equation*}
\begin{split}
&\begin{cases}
\dcap{\alpha}\s(t) = \mu-\beta \s(t)\I(t)+\gamma \I(t) -\mu \s(t)\\
\dcap{\alpha}\I(t) = \beta \s(t)\I(t)-\gamma \I(t) -\mu \I(t)
\end{cases}\\
&\text{with $\s(t)+\I(t)=N(t)$, $\s(0)=\szero$ and $\I(0)=\izero$,}
\end{split}
\end{equation*}
for the constant population case $N(t)=1$, $\forall\, t$, where $\alpha \in (0,1)$ is the order of the Caputo fractional derivative,  $\mu$ is the birth rate and the death removal rate, $\beta$ is the contact rate and $\gamma$ is the recovery removal rate. The unknown functions $S(t)$ and $I(t)$ represent the percentage of susceptible and infected people at time $t>0$ with initial data $S_0$ and $I_0$. As far as we know, although the numerical literature it is unknown a formula for the solution. By using a series representation for the solution to the fractional logistic equation we may give an explicit formula for the unknown functions $S$ and $I$. From the numerical point of view, we validate the goodness of the theoretical formulas by applying two different numerical schemes. Then, we compare the  fractional case results ($0<\alpha<1$) with the well-known standard case taking the limit $\alpha\uparrow 1$ and  we analyse the effects produced by the fractional derivatives.
 
%\begin{figure}[h!]
%\centering
%\begin{tikzpicture}
%\draw (1,3) circle (1cm) node {\large$\snc$};
%\draw [->](1.8,3.6) -- node[above = 0.05] {\small $\disp\beta \frac{\snc\Inc}{\sommanc}$}(4.2,3.6);
%\draw [->](4.2,2.4) -- node[below = 0.05] {\small$\gamma\Inc$}(1.8,2.4);
%\draw [->](1,6) -- node[right = 0.05] {\small $\Lambda\sommanc$}(1,4);
%\draw [->](1,2) -- node[right = 0.05] {\small$\mu\snc$}(1,0);
%\draw [->](5,2) -- node[right = 0.05] {\small$\mu\Inc$}(5,0);
%\draw (5,3) circle (1cm) node {\large$\Inc$};
%\end{tikzpicture}
%\caption{Graphical description of $\sis$ model.}
%\label{fig:sis}
%\end{figure}
%
%\begin{figure}[htbp!]
%\centering
%\begin{tikzpicture}
%\draw (1,3) circle (1cm) node {\large$\s$};
%\draw [->](1.8,3.6) -- node[above = 0.05] {\small $\disp\beta \s\I$}(4.2,3.6);
%\draw [->](4.2,2.4) -- node[below = 0.05] {\small$\gamma\I$}(1.8,2.4);
%\draw [->](1,6) -- node[right = 0.05] {\small $\mu$}(1,4);
%\draw [->](1,2) -- node[right = 0.05] {\small$\mu\s$}(1,0);
%\draw [->](5,2) -- node[right = 0.05] {\small$\mu\I$}(5,0);
%\draw (5,3) circle (1cm) node {\large$\I$};
%\end{tikzpicture}
%\caption{Graphical description of $\sis$ model with constant population.}
%\label{fig:sisCost}
%\end{figure}

\subsection{Motivations}

Let us consider an infective disease which does not confer immunity and which is transmitted through contact between people. We divide the population into two disjoint classes which evolve in time: the susceptibles and the infectives. The first class contains the individuals which are not yet infected but who can contract the disease; the second class contains the infected population which can transmit the disease. The $\sis$ model \cite{hethcote1989AME} is a simple disease model without immunity, where the individuals recovered from the infection come back into the class of susceptibles. Such a model is used to describe the dynamic of infections which do not confer a long immunity, as the cold or influenza. Fractional calculus is therefore considered in biological models to take into account macroscopic effect. The use of fractional derivatives in the model means that some global effect may produce slowdown in the process. This is verified and discussed in the validation of the model. 
%In Figure \ref{fig:sisCost} we draw an intuitive representation of the $\sis$ model with constant population.
%
%\begin{figure}[htbp!]
%\centering
%\includegraphics[scale=.6]{grafici/sisCost.pdf} \caption{Graphical description of $\sis$ model with constant population.}
%\label{fig:sisCost}
%\end{figure}

\subsection{State of the art}
The logistic function was introduced by Pierre Francois Verhulst \cite{verhulst1838CMP} to model the  population growth. At the beginning  of the process the growth of the population is fast;
then, as saturation process begins, the growth slows, and then growth is close to be flat. The problem to give a solution of the fractional logistic equation was unsolved and several attempts have been done (see for instance \cite{attempt2, attempt1, West, according1}). Concerning the fractional $\sis$ model, some works can be listed about numerical solutions obtained by considering different methods. From the technical point of view our result take advantage of the explicit representation  by series of the solution of fractional logistic equation solved in the recent paper  \cite{dovidio2018PA}. Thanks to a fruitful formulation of the $\sis$ model we are able to adapt  the results obtained for fractional logistic equation in \cite{dovidio2018PA} and to give the solution of fractional $\sis$ model by series.

In recent years the study of epidemiological models using fractional calculus has spread widely. In \cite{pooseh2011AIP} the authors prove via numerical simulations that the proposed fractional model gives better results than the classical theory, when compared to real data. Moreover, for some diseases it is necessary to take into account the history of the system (see for example \cite{rositch2013IJC}), thus non-locality and memory become important to model real data. Indeed, fractional operators consider the entire history of the  biological process  and we are able to model non local effects often encountered in biological phenomena. 

\subsection{Main results}
We provide an explicit representation of the solution to
\begin{equation}
\begin{split}
&\begin{cases}
\dcap{\alpha}\s(t) = \mu-\beta \s(t)\I(t)+\gamma \I(t) -\mu \s(t)\\
\dcap{\alpha}\I(t) = \beta \s(t)\I(t)-\gamma \I(t) -\mu \I(t)
\end{cases}\\
&\text{with $\s(t)+\I(t)=1$, $\s(0)=\szero$ and $\I(0)=\izero$,}
\end{split}
\label{eq:sisCap}
\end{equation}
in terms of uniformly convergent series on compact sets. 

Let us introduce the \emph{basic reproduction number} \cite{anderson1981PTRS} i.e. the expected number of secondary infections produced during the period of infection, which is given  by 
\begin{equation}\label{eq:sigma}
\sigma = \frac{\beta}{\gamma+\mu},
\end{equation} 
where $\gamma+\mu$ is the infection period. Let \begin{equation} 
\label{eq:cc}
c = \disp\sig
\end{equation}
be the so-called \emph{carrying capacity} and define $\molt=\beta c$. The problem \eqref{eq:sisCap} can be solved by considering the fractional logistic equation
\begin{equation}
\label{eqMainI}
\dcap{\alpha}\I(t) = b\, I(t) \left( 1 - \frac{1}{c} I(t) \right)
\end{equation}
In the following theorems, $B(x,y)$ denotes the Beta function, $\Gamma(x)$ denotes the Euler Gamma function and $E^\alpha_k$ are the $\alpha$-Euler's number introduced in \cite{dovidio2018PA}.

\begin{theorem}
\label{thm:solFrac}
Let $\alpha\in(0,1)$, $c\neq0$ and $b^{1/\alpha}<1$. An explicit representation of the solution of the fractional $\sis$ model \eqref{eq:sisCap} with initial condition $\izero=c/2$ and $\szero=1-\izero$ is given by
\begin{align}
\I(t) &= c\sum_{k\geq0}\ek{k}\,\molt^{\alpha k}\frac{t^{\alpha k}}{\Gamma(\alpha k+1)}\label{eq:solIfle}\\
\s(t) &= 1-I(t),\label{eq:solSfle}
\end{align}
with	
\begin{align*}
\ek{0} = \frac{1}{2}, \qquad \ek{1} = \ek{0}-(\ek{0})^{2}
\end{align*}
and $ \forall k\geq1$
\begin{align*}
\displaystyle \ek{2k} = 0, \qquad \ek{2k+1} = -\frac{1}{\alpha k+1}\sum_{\substack{i,j\\ i+j=k}}\frac{\ek{i}\ek{j}}{B(\alpha i+1,\alpha j+1)}.
\end{align*} 
The series is uniformly convergent on any compact subset $K\subseteq (0,\ralph)$, where 
\begin{equation}
r_{\alpha} = \frac{1}{\molt^{1/\alpha}}\left(\frac{\Gamma(\alpha+1)\Gamma(3\alpha+1)}{\Gamma(2\alpha+1)}\right)^{\frac{1}{2\alpha}}.
\end{equation}
\end{theorem}

\begin{theorem}
\label{thm:solFrac2}
Let $\alpha\in(0,1)$, $c=0$. An explicit representation of the solution of the fractional $\sis$ model \eqref{eq:sisCap} with initial condition $\izero=1/(2\beta)$ and $\szero=1-\izero$ is given by
\begin{align}
\I(t) &= \frac{1}{\beta}\sum_{k\geq0}\ak{k}\frac{t^{\alpha k}}{\Gamma(\alpha k+1)},\label{eq:solIfle2}\\
\s(t) &= 1-\I(t),\label{eq:solSfle2}
\end{align}
with	$\ak{0} =\frac{1}{2},\, \ak{1} = -(\ak{0})^{2}$ and 
\begin{equation*}
\ak{k+1} = -\frac{1}{\alpha k+1}\sum_{\substack{i,j\\ i+j=k}}\frac{\ak{i}\ak{j}}{B(\alpha i+1,\alpha j+1)} \quad\forall\, k\geq1.
\end{equation*} 
The series converges uniformly in $K \subset (0, r_\alpha)$ with $ r_\alpha \leq (1/2)^{1/\alpha}$.
\end{theorem}

\subsection{Outline}

The paper is organized as follows. 
In Section \ref{secSetting} we introduce the fractional $\alpha$-$\sis$ model
with  constant population size. In Section \ref{sec:proof} %we give  the briefly recall the tools of fractional calculus used in the computations, and 
we prove the main results of the paper.
In Section \ref{sec:comparison} we validate the model using  two numerical schemes  and we provide some numerical tests also comparing
the $\alpha$-$\sis$ model with the $\sis$ one.

\section{The Settings}
\label{secSetting}
 
\subsection{The fractional derivatives}
\label{sec:fracCal}

Fractional Calculus has a long history. Starting from some works by Leibniz (1695) or Abel (1823), it has been developed up to nowadays. The literature is vast and many definitions of fractional derivatives has been given. We recall the well-known derivatives of Caputo and Riemann-Liouville given by following the definitions we will deal with throughout. The Caputo Derivative of a function $u(t)$ is written as
\begin{align}
\dcap{\alpha} u(t):= \frac{1}{\Gamma(1-\alpha)} \int_0^t \frac{u^\prime(s)}{(t-s)^\alpha}ds, \quad t>0
\end{align}
whereas, the Riemann-Liouville derivative of $u(t)$ is defined as follows
\begin{equation}\label{eq:drl}
	\drl\alpha u(t) = \frac{1}{\Gamma(1-\alpha)}\frac{d}{dt}\int_{0}^{t}\frac{u(s)}{(t-s)^{\alpha}}ds.
\end{equation}
Notice that, for $a<b$, if $u \in L^1(a,b)$ such that $u^\prime \in L^1(a,b)$ and $|u^\prime (t) | \leq t^{\gamma - 1}$ a.e.  with $\gamma >0$, then we have that for $t \in (a,b)$
\begin{align*}
\big| \dcap{\alpha} u(t) \big| \leq \frac{1}{\Gamma(1-\alpha)} \int_0^t s^{\gamma - 1} (t - s)^{1-\alpha -1} ds = \frac{B(\gamma, 1-\alpha)}{\Gamma(1-\alpha)}
\end{align*} 
where
\begin{align*}
B(\alpha, \beta) = \frac{\Gamma(\alpha)\, \Gamma(\beta)}{\Gamma(\alpha+\beta)}, \quad \alpha>0, \; \beta>0
\end{align*}
is the Beta function and $ \Gamma(\alpha) = \int_0^\infty e^{-s} s^{\alpha-1} ds$, $\alpha>0$ is the Euler's gamma function. The Caputo and the Riemann-Liouville fractional derivatives are linked by the following formula
\begin{align}\label{eq:caprl}
\dcap{\alpha} u(t) = \drl{\alpha} u(t) - \frac{t^{-\alpha}}{\Gamma(1-\alpha)} u(0) = \drl{\alpha} \big( u(t) - u(0) \big)
\end{align}
which will be useful further on. We list some useful properties of the Caputo derivative:
\begin{enumerate}[label={(P\arabic*)}]
	\item\label{p0} Let $u$ be a constant function. Then $\dcap{\alpha} u(t)=0$.
	\item\label{p1} Le $u:[a,b] \to \R$ such that $u(a)=0$ and $\dcap{\alpha} u$, $\drl{\alpha} u$ exist almost everywhere. Then, $\dcap{\alpha} u = \drl{\alpha}u$.
	\item\label{p2}  Let $u,v:[a,b] \to \R$ be such that $\dcap{\alpha} u(t)$ and $\dcap{\alpha}v(t)$ exist almost everywhere in $[a,b]$. Let $c,d \in \R$. Then, $\dcap{\alpha} (c u(t) + d v(t))$ exists almost everywhere in $[a,b]$. In particular,
\begin{align*}
\dcap{\alpha} (c u(t) + d v(t)) = c \dcap{\alpha}u(t) + d \dcap{\alpha} v(t).
\end{align*} 
	\item\label{p3} Let $u \in C^1([a,b])$. Then, 
$$\dcap{\alpha} u(t) \to u^\prime(t), \quad \textrm{as} \quad \alpha \to 1^- $$
pointwise in $(a, b]$.
\end{enumerate}
\ref{p0} and \ref{p2} are immediate consequences of the definition of the Caputo derivative. \ref{p1} can be obtained from \eqref{eq:caprl}. \ref{p3} follows from the definition given for $\alpha \in (0,1)$. Our discussion here is based on the result in \cite[Theorem 2.20]{DiethBOOK} for the Riemann-Liouville derivative and the definition \eqref{eq:caprl} above of the Caputo derivative. The interested reader can also consult \cite[page 20]{Ferrari} in which the connection with the Marchaud derivative is considered. 

Let us consider the equation 
%\begin{align*}
%\mathscr{D}^\alpha_t u + a\, u=0 \quad \textrm{on $K= [0,\infty)$ with $u(0) = 1$ where $a\in\R$.}
%\end{align*}
$\mathscr{D}^\alpha_t u + a\, u=0$ on $K= [0,\infty)$ with $u(0) = 1$ where $a\in\R$. 
Then, $u$ is the Mittag-Leffler function
\begin{equation}\label{eq:mittag}
	u(t) = E_{\alpha}(-at^{\alpha}) = \sum_{k\geq0}(-a)^{k}\frac{t^{\alpha k}}{\Gamma(\alpha k +1)}, \quad t\in K.
\end{equation}
For the reader's convenience we write below the proof of this standard result. From the Laplace transform
\begin{align*}
\int_0^\infty e^{-\lambda t} \mathscr{D}^\alpha_t u(t)\, dt = \lambda^\alpha \widetilde{u}(\lambda) - \lambda^{\alpha-1}u(0)
\end{align*}
where $\widetilde{u}(\lambda) = \int_0^\infty e^{-\lambda t} u(t)dt$, the equation takes the form
$\lambda^\alpha \widetilde{u}(\lambda) - \lambda^{\alpha-1} u(0) = a\, \widetilde{u}(\lambda)$ that is
\begin{align*}
\widetilde{u}(\lambda) = u(0) \frac{\lambda^{\alpha-1}}{a + \lambda^\alpha} = \int_0^\infty e^{-\lambda t} \, E_{\alpha}(-a t^{\alpha})\, dt, \quad \lambda>0,
\end{align*}
since $u(0)=1$. From the Stirling's formula for Gamma function we have 
\begin{align*}
\left( \frac{a^k}{\Gamma(\alpha k +1)} \right)^{1/k} \sim a \left( \frac{e}{\alpha k + 1}\right)^\frac{\alpha+1}{k} \big(2\pi (\alpha k+1) \big)^{-1/(2k)} (1+ o(1)).
\end{align*}
Thus, we get that
\begin{align*}
\left( \frac{a^k}{\Gamma(\alpha k +1)} \right)^{1/k} \to 0 \quad \textrm{as} \quad k \to \infty.
\end{align*}
Thus, by the root criterion, we get an infinite radius of convergence.

\subsection{The fractional $\sis$ model}
\label{sec:FSISM}

In the discussion above the symbols $S(t)$ and $I(t)$ have been used denoting percenteges. Indeed, $N(t)=1$ is a constant function for any $t$. Denoting by $\snc(t)$ and $\Inc(t)$ the number of susceptibles and infectives, respectively, at time $t$, the fractional $\sis$ model with non constant population (see \cite{zhou1994JMB,zhang2017SAM} for $\alpha=1$, that is the non fractional case, we say SIS model) is written as 
\begin{equation}
\begin{split}
&\begin{cases}
\disp\dcap{\alpha}\snc(t) = \Lambda\sommanc(t)-\beta \frac{\snc(t)\Inc(t)}{\sommanc(t)}+\gamma \Inc(t) -\mu \snc(t)\smallskip\\
\disp\dcap{\alpha}\Inc(t) = \beta \frac{\snc(t)\Inc(t)}{\sommanc(t)}-\gamma \Inc(t) -\mu \Inc(t)
\end{cases}\\
&\text{with $\sommanc(t) = \snc(t)+\Inc(t)$, $\s(0)=\szeronc$ and $\I(0)=\izeronc$,}
\end{split}
\label{eq:sisCapNC} 
\end{equation}
where $\Lambda$ is the birth rate, $\mu$ is the death removal rate, $\beta$ is the contact rate and $\gamma$ is the recovery removal rate. The sum of susceptibles and infectives is defined by $\sommanc(t)$.

The problem to solve \eqref{eq:sisCapNC} is challenging for many reasons. To overcome such difficulties we introduce the difference between the susceptible and infective populations given by
\begin{equation}
\label{eq:ned}
	\diffnc(t) = \snc(t)-\Inc(t),
\end{equation}
from which we are able to recover the functions $\snc$ and $\Inc$ as follows
\begin{equation*}
	\snc(t) = \frac{\sommanc(t)+\diffnc(t)}{2} \qquad \textrm{and} \qquad \Inc(t)=\frac{\sommanc(t)-\diffnc(t)}{2}.
\end{equation*}
By the linearity of the Caputo derivative (see \ref{p2}) the problem takes the form
\begin{align}
	\dcap{\alpha}\sommanc(t) &= (\Lambda -\mu) \sommanc(t)\label{eq:ned2Ncap}\\
	\dcap{\alpha}\diffnc(t) &=  \left(\Lambda-\frac{\beta}{2}+\gamma\right)\sommanc(t)-(\gamma+\mu)\diffnc(t)\left(1-\frac{\beta}{2\sommanc(t)(\gamma+\mu)}\diffnc(t)\right).\label{eq:ned2Dcap}
\end{align}
In this new formulation we are able to solve \eqref{eq:ned2Ncap} by using standard results.
\begin{prop} \label{propMittag}
The solution to \eqref{eq:ned2Ncap} with initial datum $\nzeronc=\szeronc+\izeronc$ is 
\begin{equation}
\label{mlSolN}
	\sommanc(t) = \nzeronc E_{\alpha}((\Lambda-\mu) t^{\alpha}),
\end{equation}
where $E_{\alpha}$ is the Mittag-Leffler function, defined in \eqref{eq:mittag}.
\end{prop}
Notice that $\sommanc(t)\geq 0$ is an increasing function as $\Lambda-\mu>0$ whereas, it exhibits a decreasing behaviour for $\Lambda-\mu<0$. Thus, we can write the non-obvious relation
\begin{align*}
\mathscr{D}^\alpha_t \sommanc(t) >0 \quad \textrm{if $ \Lambda > \mu$ and $\sommanc(t)$ is increasing},  \\
\mathscr{D}^\alpha_t \sommanc(t) <0 \quad \textrm{if $\Lambda< \mu$ and $\sommanc(t)$ is decreasing}.
\end{align*}
We underline that the fractional derivative is a non-local operator and we do not have a direct information about the behaviour of the function under investigation.

The equation \eqref{eq:ned2Dcap} can be treated as a fractional logistic equation with a forcing term. We decided to focus on this equation in a different work. Although the problem can be studied from a numerical point of view, proceeding with a general approach seems to be hard. \\

Our results can be regarded as the special case $\Lambda=\mu$, that is constant population $\sommanc(t)$, $t>0$. Indeed, for the Mittag-Leffler function we have $E_\alpha(0)=1$, $\forall\, \alpha\in (0,1)$. Thus, we turn our problem in studying the fractional logistic equation. In particular, assuming $\Lambda = \mu$ the problem reduces to 
\begin{align}
	\dcap{\alpha}\sommanc(t) &= 0\\
	\dcap{\alpha}\diffnc(t) &=  \left(\Lambda-\frac{\beta}{2}+\gamma\right)\sommanc(t)-(\gamma+\mu)\diffnc(t)\left(1-\frac{\beta}{2\sommanc(t)(\gamma+\mu)}\diffnc(t)\right)
\end{align}
that is, $\sommanc(t)$ is constant and satisfies \ref{p0} as we can see from the first equation and the second equation is the fractional logistic equation we are interested in with the suitable characterization of all parameters. Indeed, by considering $\mathcal{N}(t)=C$ with the corresponding compartmental  $\Lambda C$, $\beta C$, $\mu C$, $\gamma C$, the equations above take the form
\begin{equation}
\begin{split}
&\begin{cases}
C \disp\dcap{\alpha} \frac{\snc(t)}{C} = \Lambda C -\beta C \frac{\snc(t)}{C} \frac{\Inc(t)}{C} + \gamma C \frac{\Inc(t)}{C} -\mu  C \frac{ \snc(t)}{C}\smallskip\\
C \disp\dcap{\alpha} \frac{\Inc(t)}{C} = \beta C \frac{\snc(t)}{C}\frac{\Inc(t)}{C}-\gamma C \frac{\Inc(t)}{C} -\mu C \frac{\Inc(t)}{C}
\end{cases}\\
&\text{with $C = \snc(t)+\Inc(t)$, $\s(0)=\szeronc$ and $\I(0)=\izeronc$,}
\end{split}
\end{equation}
and we get
\begin{equation}
\begin{split}
&\begin{cases}
C \disp\dcap{\alpha} S(t) = \Lambda C -\beta C S(t) I(t) + \gamma C I(t) -\mu  C S(t) \smallskip\\
C \disp\dcap{\alpha} I(t) = \beta C S(t) I(t) - \gamma C I(t) -\mu C I(t)
\end{cases}\\
&\text{with $1 = S(t) + I(t)$, $S(0)=S_0$ and $I(0)=I_0$,}
\end{split}
\end{equation}
where $S_0=\szeronc /C$ and $I_0=\izeronc/C$. Remember that $I(t) = \mathcal{I}(t)/C$ is a percentage, by recalling that $\mathcal{Z}(t) = C - 2 \mathcal{I}(t)$ and $\Lambda= \mu$ we obtain 
\begin{align*}
-2 \disp\dcap{\alpha} \mathcal{I}(t) 
= &  \left( \mu + \gamma - \frac{\beta}{2} \right) C - (\gamma + \mu) (C - 2 \mathcal{I}(t)) \left( 1 - \frac{\beta}{2C (\gamma + \mu)} (C - 2 \mathcal{I}(t)) \right)\\
= & - \frac{\beta}{2} C + 2 (\gamma+ \mu) \mathcal{I}(t) + \frac{\beta}{2C} (C - 2\mathcal{I}(t))^2\\
= & 2(\gamma + \mu - \beta) \mathcal{I}(t) + 2 \frac{\beta}{C} \mathcal{I}^2(t)
\end{align*}
that is
\begin{align*}
- 2 C \disp\dcap{\alpha} I(t) = - 2 (\beta - (\gamma+\mu)) C I(t) + 2 \beta C I^2(t)
\end{align*}
from which we recover
\begin{align*}
\disp\dcap{\alpha} \mathcal{I}(t) = \beta c I(t) - \beta I^2(t)
\end{align*}
which is \eqref{eqMainI}. We notice that in this characterization the carrying capacity $c$ merits further investigations. Indeed, it must be $c \neq 1$. We are lead to study both cases $c=0$ and $c\neq 0$.  Since, in our formulation, $\sommanc(t)=1$ we refer to $\snc(t)$ and $\Inc(t)$ as percentages and use the symbol $S(t)$ and $I(t)$.

For $\alpha=1$ the Mittag-Leffler becomes the exponential $E_1((\Lambda - \mu)t)=e^{(\Lambda-\mu) t}$ whereas, for $\alpha \in (0,1)$ we have the following asymptotic behaviours for $\Lambda \leq \mu$,
\begin{align*}
\frac{E_\alpha((\Lambda-\mu)t^\alpha)}{e_0((\Lambda-\mu)t^\alpha)} \to 1, \quad \textrm{as} \quad t\to 0 \quad \textrm{and} \quad  \frac{E_\alpha((\Lambda-\mu)t^\alpha)}{e_\infty((\Lambda-\mu)t^\alpha)} \to 1, \quad \textrm{as} \quad  t\to \infty
\end{align*} 
where
\begin{align*}
e_0((\Lambda-\mu)t^\alpha) = \exp\left( - |\Lambda-\mu| \frac{
t^\alpha}{\Gamma(1+\alpha)} \right), \quad \textrm{and} \quad e_\infty((\Lambda-\mu)t^\alpha) = \frac{1}{|\Lambda-\mu|} \frac{ t^{-\alpha}}{\Gamma(1-\alpha)}.
\end{align*}
For $\Lambda>\mu$, the Mittag-Leffler \eqref{mlSolN} is an increasing function.

\section{Proof of the main results}
\label{sec:proof}

In this section we collect the proof of the results presented in the work.  

From the theory of power series we know that to each series representation with coefficients $\{\psi_k\}_k$ corresponds a radius of convergence $r_\alpha \in [0, \infty]$ such that the series converges uniformly in $(0, r)$ for every $r< r_\alpha$. By the root test we also have that
\begin{align}
r_\alpha = \left( \lim_{k \to \infty} \sup \bigg| \frac{\psi_k}{\Gamma(\alpha k + 1)} \bigg|^{1/k} \right)^{-1/\alpha}
\end{align}
and the radius $r_\alpha$ obviously depends on the sequence $\{\psi_k\}_k$ and the order $\alpha \in (0,1)$ of the fractional derivative.

\begin{proof}[Proof of Theorem \ref{thm:solFrac}]
Similarly to the classical case, by the linearity \ref{p2} of the Caputo derivative, we exploit $\s(t) = 1-\I(t)$ to reduce problem \eqref{eq:sisCap} to
\begin{equation}\label{eq:iCap}
\dcap{\alpha}\I(t) = \beta c\I(t) \left(1-\frac{\I(t)}{c}\right).
\end{equation}
We rewrite \eqref{eq:iCap} as 
\begin{equation}\label{eq:fracLog}
\dcap{\alpha} v(t) = \frac{1}{M^{\alpha}}v(t)(1-v(t)),
\end{equation}
where $v(t)=\I(t)/c$ and $M=(\beta c)^{-1/\alpha} = \molt^{-1/\alpha}$.
Equation \eqref{eq:fracLog} is the fractional logistic equation investigated in \cite{dovidio2018PA} where the explicit solution is given for $M>1$ and $v(0)=1/2$ as
\begin{equation}\label{eq:solLog}
v(t) = \sum_{k\geq0}\frac{\ek{k}}{M^{\alpha k}}\frac{t^{\alpha k}}{\Gamma(\alpha k+1)}.
\end{equation}
In particular, the authors proved an estimate by below of the convergence ray $r_\alpha$. From \eqref{eq:solLog} we recover $\I(t) = c v(t)$, solution of the $\alpha-\sis$ model.
\end{proof}

\begin{proof}[Proof of Theorem \ref{thm:solFrac2}]
By the linearity \ref{p2} of the Caputo derivative and the fact that $\s(t) = 1-\I(t)$ the problem \eqref{eq:sisCap} reduces to
\begin{equation}\label{eq:iCap2}
\dcap{\alpha}\I(t) = -\beta\I^{2}(t).
\end{equation}
Setting $u(t)=\beta\I(t)$ we have that 
\begin{equation}\label{eq:fracLog2}
\dcap{\alpha} u(t) = \beta\dcap{\alpha}\I(t)=-\beta^{2}\I^2(t) = -u^{2}(t).
\end{equation}
We prove that 
\begin{equation}\label{eq:solSig1}
	u(t) = \sum_{k=0}^{\infty}\ak{k}\frac{t^{\alpha k}}{\Gamma(\alpha k+1)}
\end{equation}
solves \eqref{eq:fracLog2}, hence $\I(t)=u(t)/\beta$ is the solution to \eqref{eq:iCap2}.

To this end we compute the Riemann-Liouville fractional derivative of $u(t)$ in \eqref{eq:solSig1} which is 
\begin{align*}
	\drl{\alpha}u(t) &= \sum_{k=0}^{\infty}\ak{k}\frac{t^{\alpha k-\alpha}}{\Gamma(\alpha k-\alpha+1)}\\
	& = \ak{0}\frac{t^{-\alpha}}{\Gamma(1-\alpha)}+\sum_{k=0}^{\infty}\ak{k+1}\frac{t^{\alpha k}}{\Gamma(\alpha k+1)}\\
	& = \ak{0}\frac{t^{-\alpha}}{\Gamma(1-\alpha)}+\ak{1}+\ak{2}\frac{t^{\alpha}}{\Gamma(\alpha+1)}+\ak{3}\frac{t^{2\alpha}}{\Gamma(2\alpha+1)}+\ak{4}\frac{t^{3\alpha}}{\Gamma(3\alpha+1)}+\ak{5}\frac{t^{4\alpha}}{\Gamma(4\alpha+1)}+\dots.
\end{align*}
By \eqref{eq:caprl}, we have 
\begin{equation}\label{eq:coeff1}
	\dcap{\alpha}u(t) = \ak{1}+\ak{2}\frac{t^{\alpha}}{\Gamma(\alpha+1)}+\ak{3}\frac{t^{2\alpha}}{\Gamma(2\alpha+1)}+\ak{4}\frac{t^{3\alpha}}{\Gamma(3\alpha+1)}+\ak{5}\frac{t^{4\alpha}}{\Gamma(4\alpha+1)}+\dots.
\end{equation}
Now we compute $u^{2}(t)$
\begin{equation}\label{eq:coeff2}
\begin{split}
	u^{2}(t) &= \sum_{k=0}^{\infty}\sum_{s=0}^{\infty}\ak{k}\ak{s}\frac{t^{\alpha (k+s)}}{\Gamma(\alpha k+1)\Gamma(\alpha s+1)}\\
	& = \ak{0}\ak{0}\\
	&+ \frac{2\ak{1}\ak{0}}{\Gamma(\alpha+1)}t^{\alpha}\\
	&+ \left(\frac{\ak{1}\ak{1}}{\Gamma(\alpha+1)\Gamma(\alpha+1)}+\frac{2\ak{0}\ak{2}}{\Gamma(2\alpha+1)}\right)t^{2\alpha}\\
	&+ \left(\frac{\ak{1}\ak{2}}{\Gamma(\alpha+1)\Gamma(2\alpha+1)}+\frac{2\ak{0}\ak{3}}{\Gamma(3\alpha+1)}\right)t^{3\alpha}\\
	&+ \left(\frac{\ak{2}\ak{2}}{\Gamma(2\alpha+1)\Gamma(2\alpha+1)}+\frac{2\ak{1}\ak{3}}{\Gamma(\alpha+1)\Gamma(3\alpha+1)}+\frac{2\ak{0}\ak{4}}{\Gamma(4\alpha+1)}\right)t^{4\alpha}+\dots
\end{split}
\end{equation}
By \eqref{eq:coeff1} and \eqref{eq:coeff2} and by $\ak{0}=1/2$ we have
\begin{align*}
	\ak{1} & = -\ak{0}\ak{0} = -1/4\\
	\ak{2} & = -2\ak{1}\ak{0}\frac{\Gamma(\alpha+1)}{\Gamma(\alpha+1)}\\
	\ak{3} & = \ak{1}\ak{1}\frac{\Gamma(2\alpha+1)}{\Gamma(\alpha+1)\Gamma(\alpha+1)}+2\ak{0}\ak{2}\frac{\Gamma(2\alpha+1)}{\Gamma(2\alpha+1)}\\
	\ak{4} & = \ak{1}\ak{2}\frac{\Gamma(3\alpha+1)}{\Gamma(\alpha+1)\Gamma(2\alpha+1)}+2\ak{0}\ak{3}\frac{\Gamma(3\alpha+1)}{\Gamma(3\alpha+1)}\\
	\ak{5}& = \ak{2}\ak{2}\frac{\Gamma(4\alpha+1)}{\Gamma(2\alpha+1)\Gamma(2\alpha+1)}+2\ak{1}\ak{3}\frac{\Gamma(4\alpha+1)}{\Gamma(\alpha+1)\Gamma(3\alpha+1)}+2\ak{0}\ak{4}\frac{\Gamma(4\alpha+1)}{\Gamma(4\alpha+1)},
\end{align*}
and thus 
\begin{equation}
	\ak{k+1} = -\sum_{j=0}^{k}\frac{\Gamma(k\alpha+1)}{\Gamma((k-j)\alpha+1)\Gamma(j\alpha+1)}\ak{j}\ak{k-j}.
\end{equation}
%\textcolor{purple}{We use the fact that $\forall\, k \in \{0,1,\ldots, \}$,
%\begin{align*}
%\frac{\Gamma(k\alpha + 1)}{\Gamma((k-j)\alpha+1)\Gamma(j\alpha+1)} =:R_k \leq \Gamma(k\alpha + 1).
%\end{align*}
%From the definition above of the coefficients $A^\alpha_k$ we get 
%\begin{align*}
%\bigg|\frac{A^\alpha_{k+1}}{\Gamma((k+1)\alpha+1)} \bigg|
%\leq &\, \frac{1}{\Gamma((k+1)\alpha + 1)}\sum_{j=0}^k R_k \big|\ak{j}\ak{k-j}\big|\\
%\leq & \, \frac{\Gamma(k\alpha+1)}{\Gamma((k+1)\alpha + 1)}\sum_{j=0}^k  \big|\ak{j}\ak{k-j} \big|.
%\end{align*}
%}
We use the fact that $\forall\, k \in \{0,1,\ldots, \}$,
\begin{align*}
\frac{\Gamma(k\alpha + 1)}{\Gamma((k-j)\alpha+1)\,\Gamma(j\alpha+1)} =:R_k \leq \Gamma(k\alpha + 1).
\end{align*}
From the definition above of the coefficients $\{A^\alpha_k\}_k$ we get 
\begin{align*}
\bigg|\frac{A^\alpha_{k+1}}{\Gamma((k+1)\alpha+1)} \bigg|
\leq &\, \frac{1}{\Gamma((k+1)\alpha + 1)}\sum_{j=0}^k R_j \big|A^\alpha_j \, A^\alpha_{k-j} \big|\\
\leq & \, \frac{\Gamma(k\alpha + 1)}{\Gamma((k+1)\alpha + 1)}\sum_{j=0}^k  \big|A^\alpha_j \, A^\alpha_{k-j} \big|.
\end{align*}
By iteration we obtain that $A^\alpha_k \sim |A^\alpha_0|^k$. Since $(0,1) \ni A^\alpha_0 \leq 1/A^\alpha_0$ we write
\begin{align*}
\bigg|\frac{A^\alpha_{k+1}}{\Gamma((k+1)\alpha+1)} \bigg|
\leq & \,\frac{\Gamma(k\alpha + 1)}{\Gamma((k+1)\alpha + 1)} (k+1) \left( \frac{1}{A^\alpha_0}\right)^k =: \vartheta_{k}, \quad k \in \mathbb{N}_0.
\end{align*}
We now consider the fact that
\begin{align*}
\frac{x^{x-\gamma}}{e^{x-1}} < \Gamma(x) < \frac{x^{x-1/2}}{e^{x-1}}, \quad x>1
\end{align*}
(where $\gamma \approx 0.5$ is the Mascheroni constant) and we get 
\begin{align*}
\sqrt[k]{ |\vartheta_k| } 
\sim \frac{1}{|A^\alpha_0|}\, \left( (k+1)\frac{(k\alpha+1)^{k\alpha + 1/2}}{((k+1)\alpha+1)^{(k+1)\alpha +1 -\gamma}} \right)^{1/k} .
\end{align*}
Since
\begin{align*}
(k\alpha +1)^{\frac{1}{k}(k\alpha +1/2)} \sim & \exp \left( \big( \alpha  + \frac{1}{2k} \big) \ln (k\alpha + 1 ) \right)
\end{align*}
and
\begin{align*}
((k+1)\alpha +1)^{\frac{1}{k}((k+1)\alpha +1-\gamma)} \sim \exp\left( \big( \alpha + \frac{1-\gamma}{k} \big) \ln ((k+1)\alpha + 1) \right)
\end{align*}
we get that
\begin{align*}
\sqrt[k]{ |\vartheta_k| } 
\sim \frac{1}{|A^\alpha_0|}.
\end{align*}
Thus, we get the radius of convergence
\begin{align*}
r^\vartheta_\alpha = \left( \lim_{k \to \infty} \big| \vartheta_{k} \big|^{1/k} \right)^{-1/\alpha} = \left( |A^\alpha_0| \right)^{1/\alpha}
\end{align*} 
for the series 
\begin{align*}
\sum_{k \geq 0} \vartheta_k.
\end{align*}
The convergence of the majorant series determines the uniform convergence in $(0, r_\alpha) \subset (0, r^\vartheta_\alpha)$ of the series we are interested in. This concludes the proof by considering $I=u/\beta$.
\end{proof}

\begin{remark}
The solution in Theorem \ref{thm:solFrac} has been given only for the initial datum $c/2$. This is because of the representation given in \cite{dovidio2018PA} in terms of Euler polynomials. Taking $A^\alpha_0 \in (0,1)$ we see that, setting 
\begin{align*}
v(t) = u(t/2^q) = \sum_{n \geq 0} A^\alpha_k \frac{(t/2^q)^{n\alpha}}{\Gamma(n\alpha+1)}, \quad t \in K^q \subseteq (0, r^q_\alpha)
\end{align*}
where
\begin{align*}
q = \left\lbrace
\begin{array}{ll}
\displaystyle \frac{1}{A^\alpha_0}, & A^\alpha_0 < \frac{1}{2}\smallskip\\
\displaystyle 4 + \frac{1}{2}\left(\frac{1}{A^\alpha_0} - 4\right), & A^\alpha_0\geq \frac{1}{2}
\end{array}
\right. 
\end{align*}
we obtain $r^q_\alpha = 2^q \left(|A_0^\alpha| \right)^{1/\alpha}$. This is the solution in $(0, r^q_\alpha)$ to 
\begin{align*}
\mathscr{D}^\alpha_t v =  - \frac{1}{2^q} v^2, \quad v(0)=A^\alpha_0 \in (0,1)
\end{align*}
(see the proof of Theorem 3.1 in \cite{dovidio2018PA}). In the special case $\alpha=1$ we know that 
\begin{align*}
w(t) = \left( \frac{1}{A_0} - t\right)^{-1}  = A_0 \sum_{k \geq 0} (-A_0)^k t^k \quad t \in (0, 1/A_0) 
\end{align*}  
solves $w^\prime = - w^2$ with $w(0)=A_0 \in (0,1)$. In particular, for $A^\alpha_0=A_0=1/2$ we obtain convergence in any compact sets $K \subset (0, 2)$ for both solutions $v$ and $w$.  
This underlines the fact that introducing non-locality we may deal with solutions quite far from their non-linear analogues.
\end{remark}

\section{Numerical comparison}\label{sec:comparison}
In this section we proceed with the validation of the previous results on the fractional $\sis$ model by means of numerical approximations, and we analyse the effects of fractional derivatives by comparing the ordinary and fractional $\sis$ model.

\subsection{Numerical approximation}\label{sec:Num}
The explicit solution \eqref{eq:solIfle}-\eqref{eq:solSfle} to the fractional $\sis$ model \eqref{eq:sisCap} for $c\neq0$ is defined for $b^{1/\alpha}<1$ and initial datum $\izero=c/2$. The explicit solution \eqref{eq:solIfle2}-\eqref{eq:solSfle2} to the fractional $\sis$ model \eqref{eq:sisCap} for $c=0$ is defined for the initial datum $\izero=1/(2\beta)$. In order to compute the solution to the fractional $\sis$ model for any set of parameters and any initial datum we propose and compare two numerical schemes to approximate \eqref{eq:sisCap}.
To this end, let us consider the following problem
\begin{equation}\label{eq:sisVett}
\dcap{\alpha}u(t) = f(u(t))\\
\end{equation}
on a time interval $[0, T]$ uniformly divided into $\nt+1$ time steps of length $\dt$. Our aim is to define the discrete solution $u_{n} = u(t_{n})$ for $n=1,\dots,N$, where $t_{n}=n\dt$ and $u_{0}$ is known.

\medskip
We refer to the following method as the Method 1. Following \cite{ahmed2005IJMP}, we observe that 
\begin{align*}
\rli{1-\alpha} u^{\prime}&=f(u)\\
\rli\alpha\rli{1-\alpha} u^{\prime}&=\rli\alpha f(u)\\
\rli{1}u^{\prime}&=\rli\alpha f(u),
\end{align*}
and thus we rewrite \eqref{eq:sisVett} as
\begin{equation}\label{eq:SISschema}
u(t) = u(0)+\rli\alpha f(u).
\end{equation}
%We observe that equations \eqref{eq:derInt} and \eqref{eq:sisVett} imply
%\begin{align*}
%\rli{1-\alpha} u^{\prime}&=f(u)\\
%\rli\alpha\rli{1-\alpha} u^{\prime}&=\rli\alpha f(u)\\
%\rli{1}U^{\prime}&=\rli\alpha f(u).
%\end{align*}
%Following \cite{ahmed2005IJMP}, we rewrite \eqref{eq:sisVett} as
%\begin{equation}\label{eq:SISschema}
%u(t) = u(0)+\rli\alpha f(u).
%\end{equation}
%To numerically solve \eqref{eq:SISschema}, we divide the time interval into $\nt+1$ time steps of length $\dt$, with $t_{0}=0$ and $t_{\nt+1}=T$. 
We introduce a Predictor-Evaluate-Corrector-Predictor (PECE) method \cite{diethelm1999FWR}. Specifically, we use the implicit one-step Adams-Moulton method \cite[Chapter 11]{quarteroni2007TAM}, i.e.
\begin{equation}\label{saka}
u_{n+1}=u_{0}+\frac{1}{\Gamma(\alpha)}\left(\sum_{j=0}^{n}a_{j,n+1}f(u_{j})+a_{n+1,n+1}f(\widetilde u_{n+1})\right),
\end{equation}
where the coefficients $a_{j,n+1}$ and $\widetilde u_{n+1}$ are defined below. \\

First of all, we compute the term $\widetilde u_{n+1}$ with the one-step Adams-Bashforth method. We introduce $g(s)=f(u(s))$ and $g_{n+1}$ as a piecewise linear function which interpolates $g$ on the nodes $t_{j}$, $j = 0,\dots,n+1$. We approximate the integral term of \eqref{eq:SISschema}  with the product rectangle rule, i.e.
\begin{equation*}
\int_{\tzero}^{\tn}(\tn-s)^{\alpha-1}g(s)ds \approx \sum_{j=0}^{n}b_{j,n+1}g(t_{j}),
\end{equation*}
where 
\begin{equation*}
b_{j,n+1}=\int_{t_{j}}^{t_{j+1}}(\tn-s)^{\alpha-1}ds=\frac{1}{\alpha}((\tn-t_{j})^{\alpha}-(\tn-t_{j+1})^{\alpha}).
\end{equation*}
In particular, for our uniform discretization of the time interval $[0,T]$, we have
\begin{equation*}
b_{j,n+1}=\frac{\dt^{\alpha}}{\alpha}((n+1-j)^{\alpha}-(n-j)^{\alpha}).
\end{equation*}
Therefore, 
\begin{equation}
\widetilde u_{n+1} = u_{0}+\frac{1}{\Gamma(\alpha)}\sum_{j=0}^{n}b_{j,n+1}f(u_{j}).
\end{equation}

Now we compute the coefficients $a_{j,n+1}$, thus we approximate $\rli\alpha g$ as
\begin{equation*}\label{eq:trap}
\int_{\tzero}^{\tn}(\tn-s)^{\alpha-1}g(s)ds\approx  \int_{\tzero}^{\tn}(\tn-s)^{\alpha-1}g_{n+1}(s)ds.
\end{equation*}
By using the product trapezoidal quadrature formula on the nodes $t_{j}$, equation \eqref{eq:trap} becomes
\begin{equation*}
\int_{\tzero}^{\tn}(\tn-s)^{\alpha-1}g_{n+1}(s)ds = \sum_{j=0}^{n+1}a_{j,n+1}g(t_{j}),
\end{equation*}
where $a_{j,n+1}$ are defined as
\begin{equation*}
a_{j,n+1} = \int_{t_{j-1}}^{t_{j}}\frac{s-t_{j-1}}{t_{j}-t_{j-1}}(\tn-s)^{\alpha-1}ds+\int_{t_{j}}^{t_{j+1}}\frac{t_{j+1}-s}{t_{j+1}-t_{j}}(\tn-s)^{\alpha-1}ds.
\end{equation*}
We observe that, from integration by parts, we have
\begin{align*}\label{eq:jm1j}
\int_{t_{j-1}}^{t_{j}}\frac{s-t_{j-1}}{t_{j}-t_{j-1}}(\tn-s)^{\alpha-1}ds  &= -\frac{(\tn-t_{j})^{\alpha}}{\alpha}+ \int_{t_{j-1}}^{t_{j}}\frac{(\tn-s)^{\alpha}}{\alpha(t_{j}-t_{j-1})}ds\\
\int_{t_{j}}^{t_{j+1}}\frac{t_{j+1}-s}{t_{j+1}-t_{j}}(\tn-s)^{\alpha-1}ds & = \frac{(\tn-t_{j})^{\alpha}}{\alpha}- \int_{t_{j}}^{t_{j+1}}\frac{(\tn-s)^{\alpha}}{\alpha(t_{j+1}-t_{j})}ds,
\end{align*}
and therefore
\begin{align*}
a_{0,n+1} &= \frac{(\tn-\tzero)^{\alpha}}{\alpha}- \int_{t_{0}}^{t_{1}}\frac{(\tn-s)^{\alpha}}{\alpha(t_{1}-t_{0})}ds\\
a_{n+1,n+1} &= \int_{t_{n}}^{t_{n+1}}\frac{(t_{n+1}-s)^{\alpha}}{\alpha(t_{n+1}-t_{n})}ds\\
a_{j,n+1} & = \int_{t_{j-1}}^{t_{j}}\frac{(\tn-s)^{\alpha}}{\alpha(t_{j}-t_{j-1})}ds -
\int_{t_{j}}^{t_{j+1}}\frac{(\tn-s)^{\alpha}}{\alpha(t_{j+1}-t_{j})}ds \qquad \text{for $j=1,\dots,n$}.
\end{align*}

Finally, in our uniform grid, the coefficients are
\begin{align}
a_{0,n+1} &= \frac{\dt^{\alpha}}{\alpha(\alpha+1)}(n^{\alpha+1}-(n-\alpha)(n+1)^{\alpha})\\
a_{n+1,n+1} &= \frac{\dt^{\alpha}}{\alpha(\alpha+1)}\\
a_{j,n+1} & = \frac{\dt^{\alpha}}{\alpha(\alpha+1)}((n-j+2)^{\alpha+1}-2(n-j+1)^{\alpha+1}+(n-j)^{\alpha+1}) \qquad \text{for $j=1,\dots,n$}.
\end{align}
\begin{remark}
The numerical scheme described above works for any $\alpha\in[0,1]$.
\end{remark}

%%%%%%%%%%%%%%%%%%%%%%%%%%%%%%
%\subsection{Second numerical scheme}\label{sec:giga}

We now introduce a method to which we refer as Method 2. Let $\alpha\in(0,1)$. In \cite{giga1902AA} the authors give the following approximation of the Caputo derivative 
%\begin{equation}
%	\dcap{\alpha}u(t_{n})=\frac{1}{\Gamma(1-\alpha)}\sum_{k=0}^{n-1}\int_{kh}^{(k+1)h}\frac{u^{\prime}(s)}{(t_{n}-s)^{\alpha}}ds,
%\end{equation}
%which leads to the following discrete version of $\dcap{\alpha}$ 
\begin{equation}
	\dcap{\alpha}u_{n}=\frac{1}{\Gamma(2-\alpha)\dt^{\alpha}}\left(u_{n}-\sum_{j=0}^{n-1}C_{n,j}u_{j} \right),
\end{equation}
with 
\begin{align*}
	C_{n,0} = g(n),\qquad C_{n,j} = g(n-j)-g(n-(j-1)) \quad\text{ for $j=1,\dots,n-1$}
\end{align*}
and $g(r) = r^{1-\alpha}-(r-1)^{1-\alpha} \text{ for $r\geq1$}$. The numerical scheme to solve \eqref{eq:sisVett} is then given by
\begin{equation}\label{giga}
	u_{n+1} = \sum_{j=0}^{n-1}C_{n,j}u_{j}+\Gamma(2-\alpha)\dt^{\alpha}f(u_{n}).
\end{equation} 
We refer to \cite{giga1902AA} for further details on the properties of the scheme.
\begin{remark}\label{rem:valoriEstremi}
The numerical scheme above described works for $\alpha\in(0,1)$, with the extreme values excluded.
\end{remark}

\medskip
To summarize, in this section we have introduced two numerical schemes which we denote here by $\numuno$ and $\numdue$ for notational convenience. The solution to the fractional $\sis$ model \eqref{eq:sisCap} with the first numerical scheme (that is Method 1) is
\begin{align}
	I(t_{n+1})&=\numuno(I(t_{n}))\label{eq:saka1}\\
	S(t_{n+1})&=1-I(t_{n+1}),\label{eq:saka2}
\end{align}
where $\numuno$ is defined in \eqref{saka}, and the solution with the second numerical scheme (that is Method 2) is 
\begin{align}
	I(t_{n+1})&=\numdue(I(t_{n}))\label{eq:giga1}\\
	S(t_{n+1})&=1-I(t_{n+1}),\label{eq:giga2}
\end{align}
where $\numdue$ is defined in \eqref{giga} and $n=1,\dots,\nt$. Note that the function $f(u)$ in \eqref{eq:sisVett}, used for both the numerical schemes, is defined as $f(u)=\beta c u - \beta u^2$, while $u_{0} = \izero$.

%%%%%%%%%%%%%%%%%%%%%%%%%%%%%%%%%%%
\subsection{Numerical tests}\label{sec:test}
In this section we compare the solutions to the fractional $\sis$ model \eqref{eq:sisCap} computed with the explicit representation and the two numerical schemes, testing both the case $c\neq0$ and $c=0$. In what follows, we denote by 
\begin{itemize}
	\item $\I^{C},\,\s^{C}$ the solutions to the $\sis$ model, our aim is to show the correspondence with the case $\alpha=1$,
	\item $\I^{F},\,\s^{F}$ the solutions \eqref{eq:solIfle}-\eqref{eq:solSfle} or \eqref{eq:solIfle2}-\eqref{eq:solSfle2} to the fractional $\sis$ model \eqref{eq:sisCap} defined by Theorems \ref{thm:solFrac} or \ref{thm:solFrac2} respectively (depending on the carrying capacity $c$),
	\item $\I_{1}^{N},\,\s_{1}^{N}$ the numerical solutions \eqref{eq:saka1}-\eqref{eq:saka2} computed with the methodology proposed as Method 1,
	\item $\I_{2}^{N},\,\s_{2}^{N}$ the numerical solutions \eqref{eq:giga1}-\eqref{eq:giga2} computed with the methodology proposed as Method 2.
\end{itemize}

\subsubsection{Test with $c\neq0$} 
We start our numerical analysis with the case of carrying capacity $c\neq0$. We fix this set of parameters: $\beta = 0.7$, $\gamma = 0.05$, $\mu=0.12$, $\sigma=4$ and $c = 0.75$. The initial data are $\I(0)=c/2$ and $\s(0)=1-\I(0)$, the final time is $T=5$ and the time step $\dt = 0.05$.

First of all we compare the exact fractional solutions \eqref{eq:solIfle}-\eqref{eq:solSfle} and the two numerical solutions \eqref{eq:saka1}-\eqref{eq:saka2} and \eqref{eq:giga1}-\eqref{eq:giga2} for $\alpha=0.99$, which approximately corresponds to the classical derivative. Note that we do not use $\alpha\equiv1$ since the second numerical scheme works for $\alpha\in(0,1)$, as already observed in Remark \ref{rem:valoriEstremi}. In Figure \ref{fig:alpha1} we show the results. As expected, the exact fractional solution and the two numerical solutions to \eqref{eq:sisCap} overlap the solution for $\alpha=1$.  
\begin{figure}[h!]
\centering
\subfloat[][Fractional solutions \eqref{eq:solIfle}-\eqref{eq:solSfle}.]{
\includegraphics[width=0.31\columnwidth]{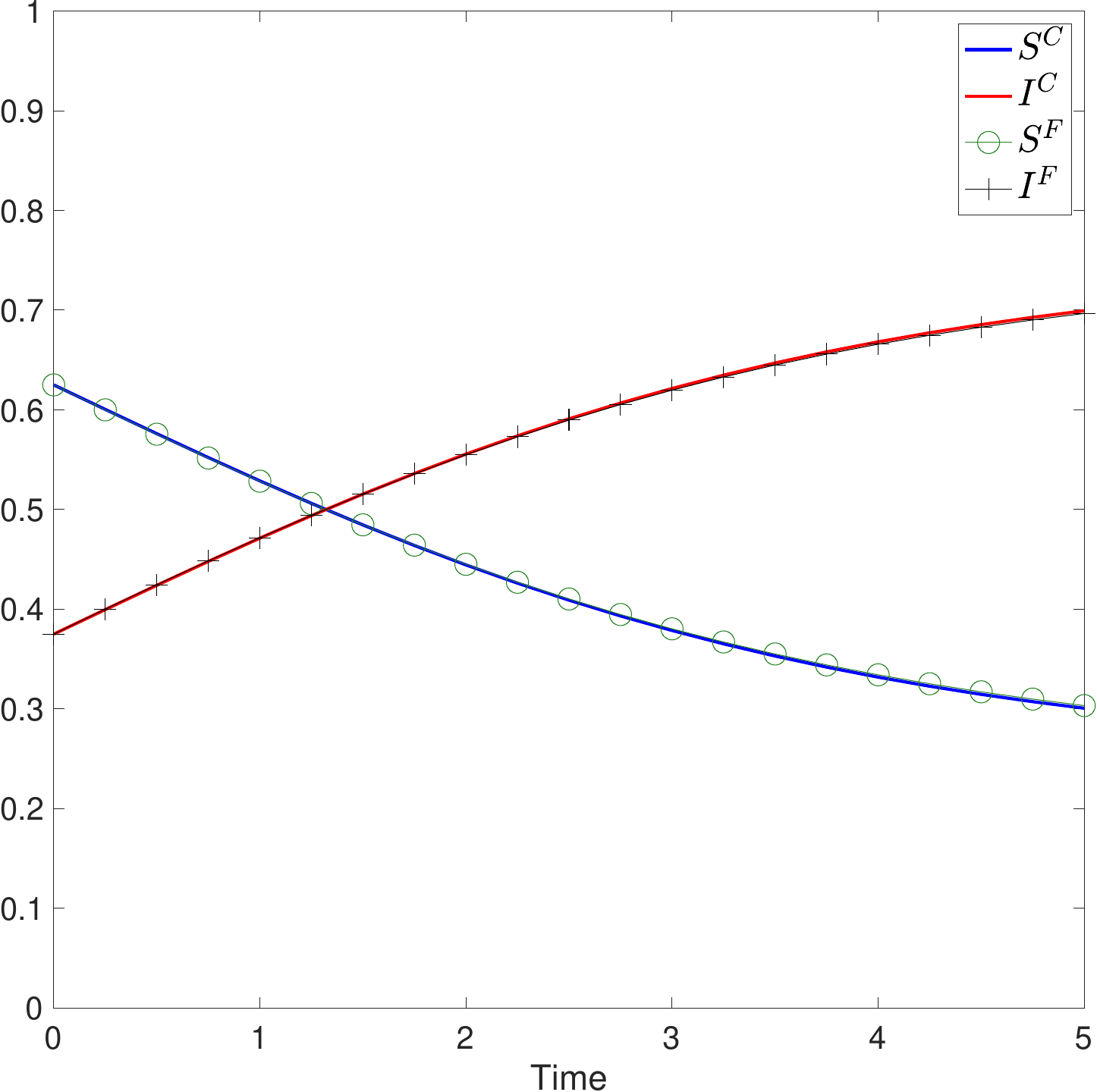}
}\,
\subfloat[][Numerical solutions \eqref{eq:saka1}-\eqref{eq:saka2}.]{
\includegraphics[width=0.31\columnwidth]{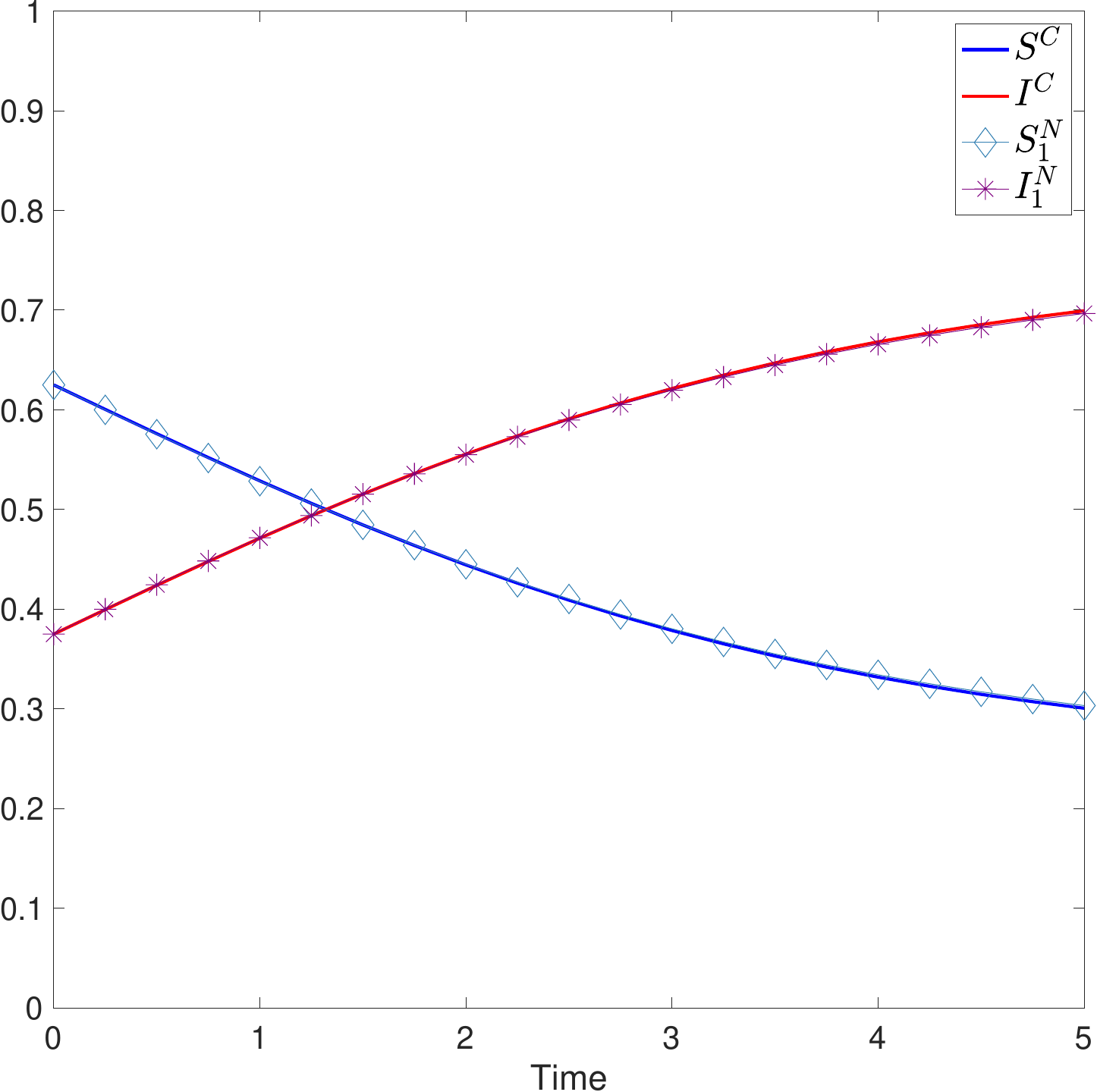}
}\,
\subfloat[][Numerical solutions \eqref{eq:giga1}-\eqref{eq:giga2}.]{
\includegraphics[width=0.31\columnwidth]{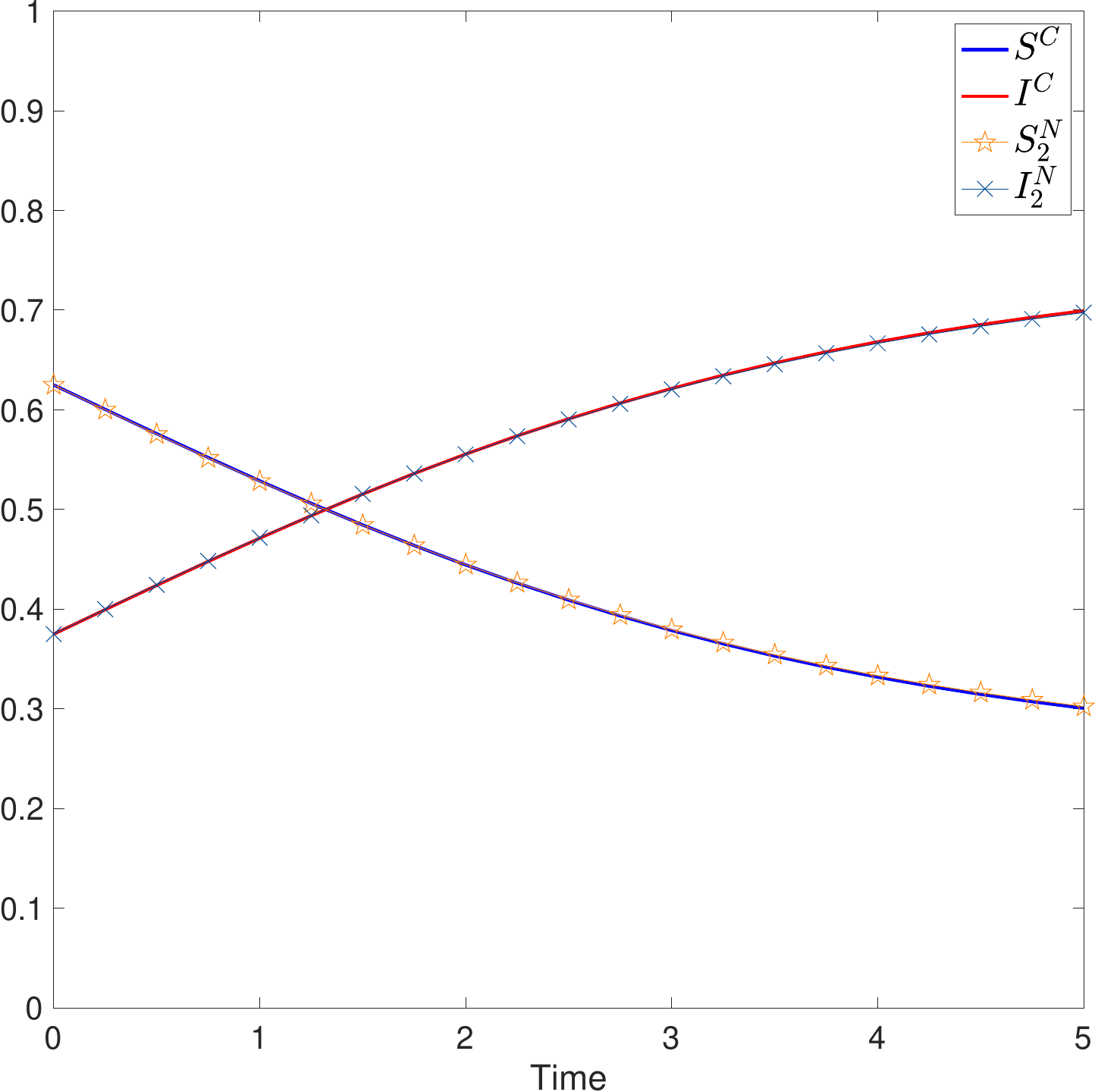}}
\caption{Comparison between the solutions to the SIS model and the explicit and numerical fractional solutions to \eqref{eq:sisCap} with $\alpha=0.99$. The analysis shows correspondence between SIS model and the case $\alpha=1$ of our model. This result was expected and it confirms the continuity wit respect to $\alpha$ (see \ref{p3}). }
\label{fig:alpha1}
\end{figure}

In Figures \ref{fig:alpha07} and \ref{fig:alpha03} we show the results obtained with $\alpha=0.7$ and $\alpha =0.3$. In the first case the two density curves are closer each other and the intersection point between them slightly moves to the right with respect to the solution shown in Figure \ref{fig:alpha1}. Such behavior is further emphasized by lower values of $\alpha$, as shown for example in Figure \ref{fig:alpha03}. Note that, in both cases the three methodologies produces almost identical results.

\begin{figure}[h!]
\centering
\subfloat[][Fractional solutions \eqref{eq:solIfle}-\eqref{eq:solSfle}.]{
\includegraphics[width=0.31\columnwidth]{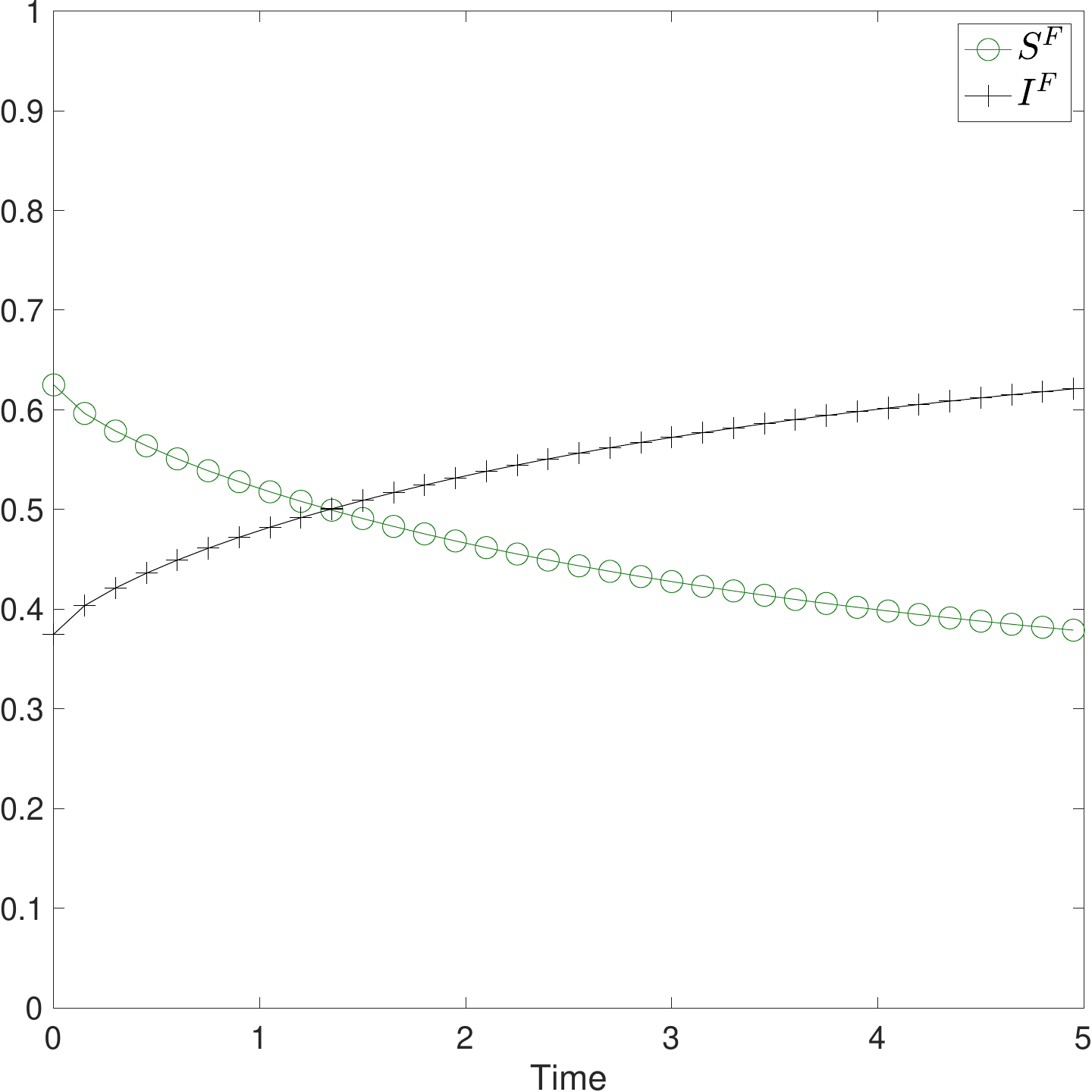}
}\,
\subfloat[][Numerical solutions \eqref{eq:saka1}-\eqref{eq:saka2}.]{
\includegraphics[width=0.31\columnwidth]{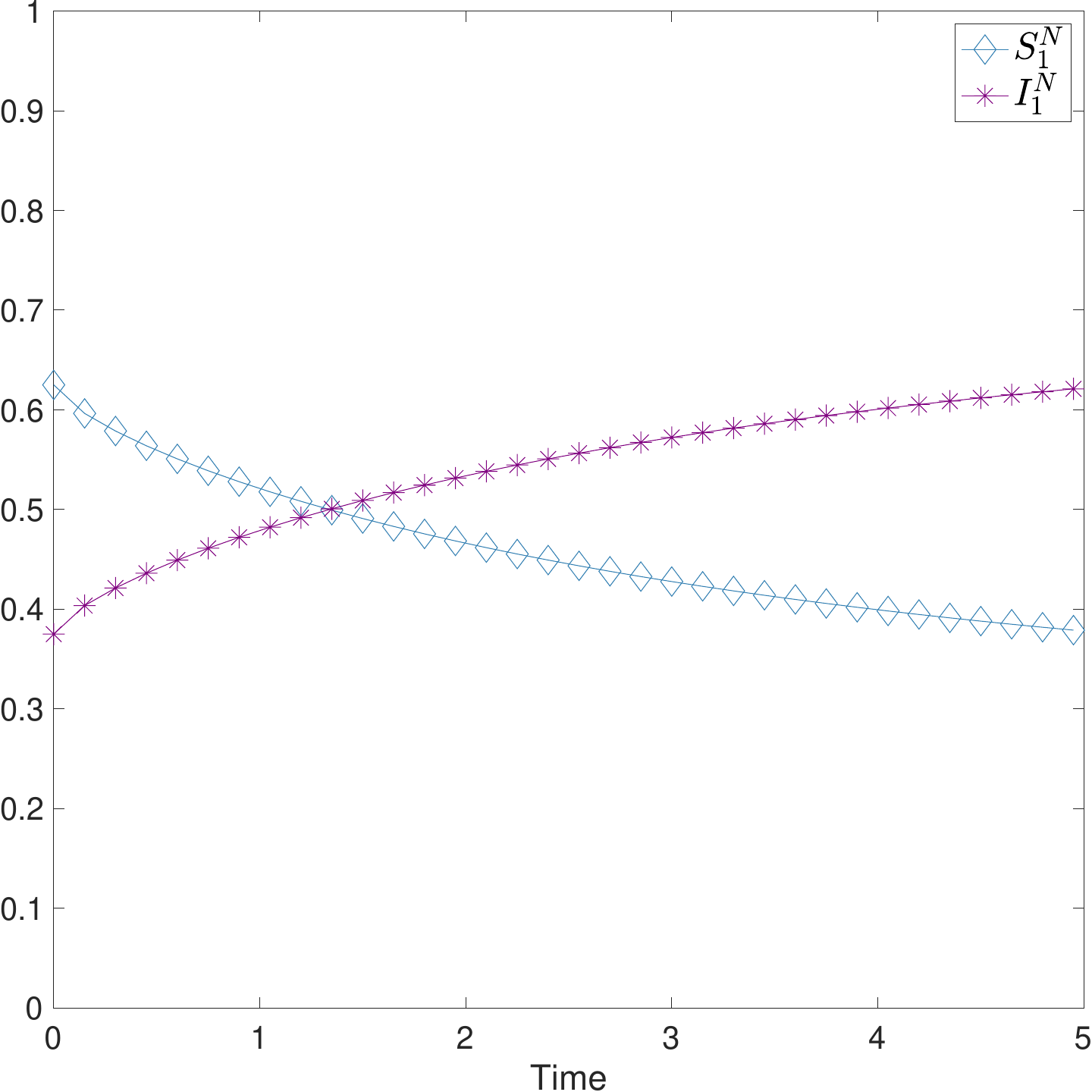}
}\,
\subfloat[][Numerical solutions  \eqref{eq:giga1}-\eqref{eq:giga2}.]{
\includegraphics[width=0.31\columnwidth]{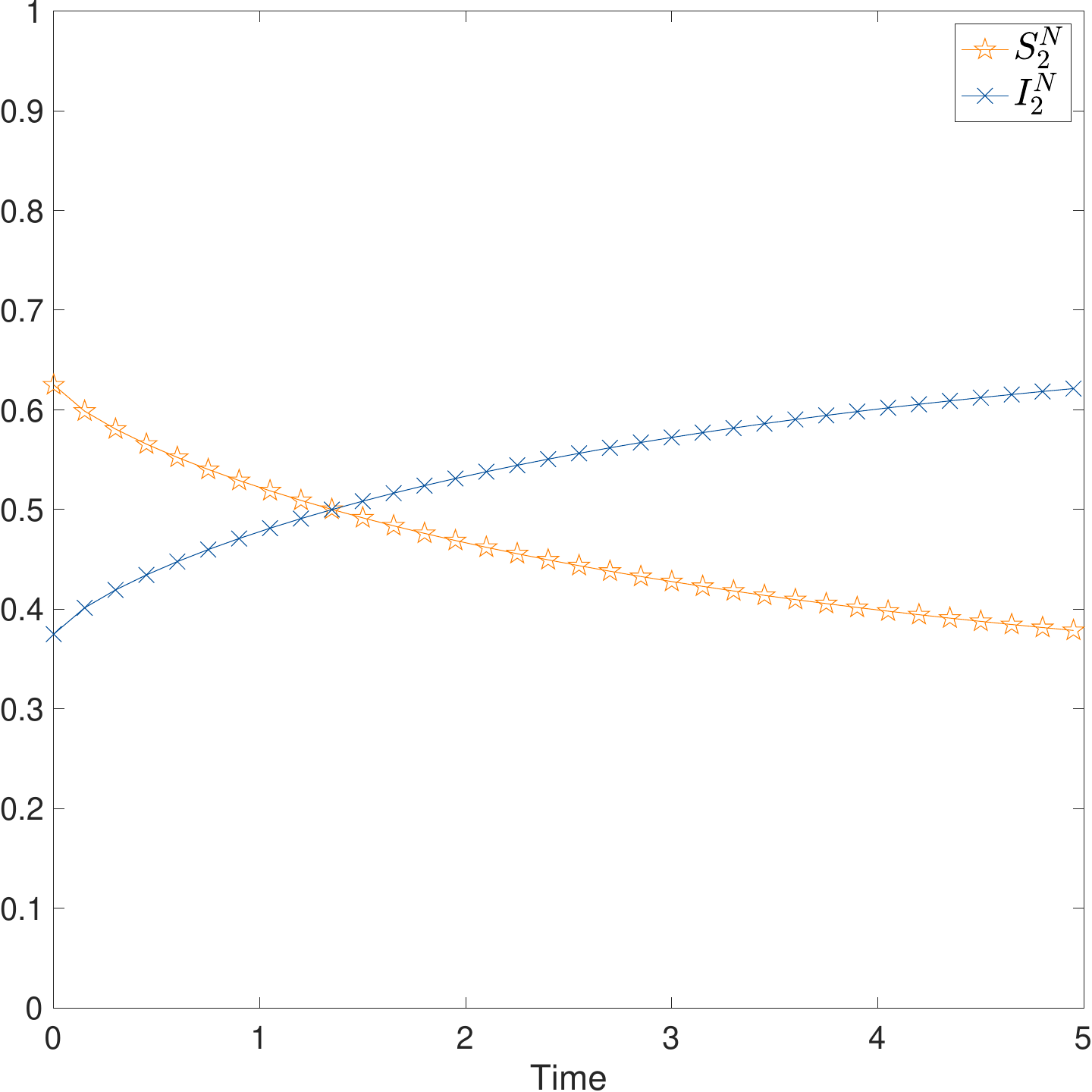}}
\caption{Comparison between the explicit and numerical fractional solutions to \eqref{eq:sisCap} with $\alpha=0.7$.}
\label{fig:alpha07}
\end{figure}

\begin{figure}[h!]
\centering
\subfloat[][Fractional solutions \eqref{eq:solIfle}-\eqref{eq:solSfle}.]{
\includegraphics[width=0.31\columnwidth]{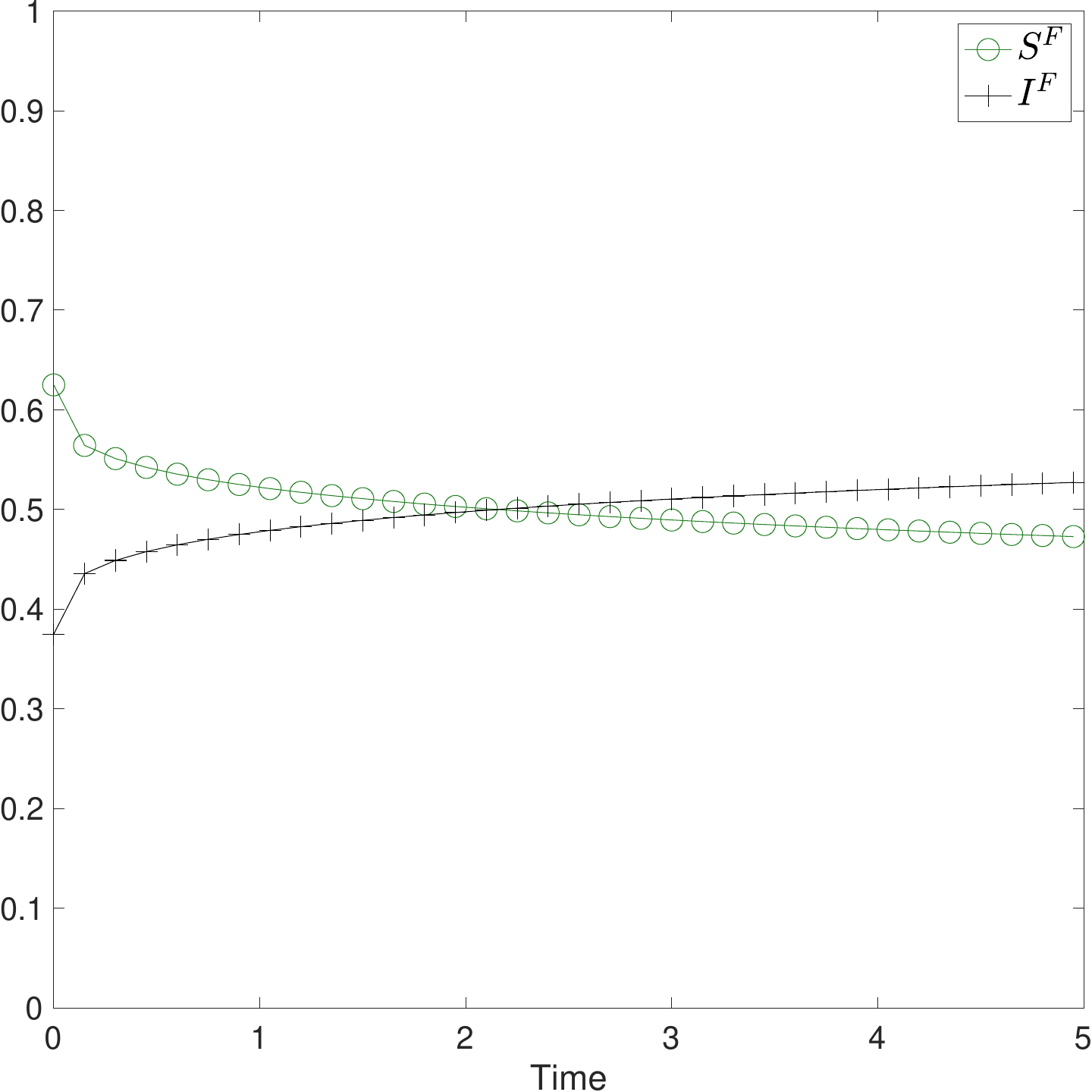}
}\,
\subfloat[][Numerical solutions \eqref{eq:saka1}-\eqref{eq:saka2}.]{
\includegraphics[width=0.31\columnwidth]{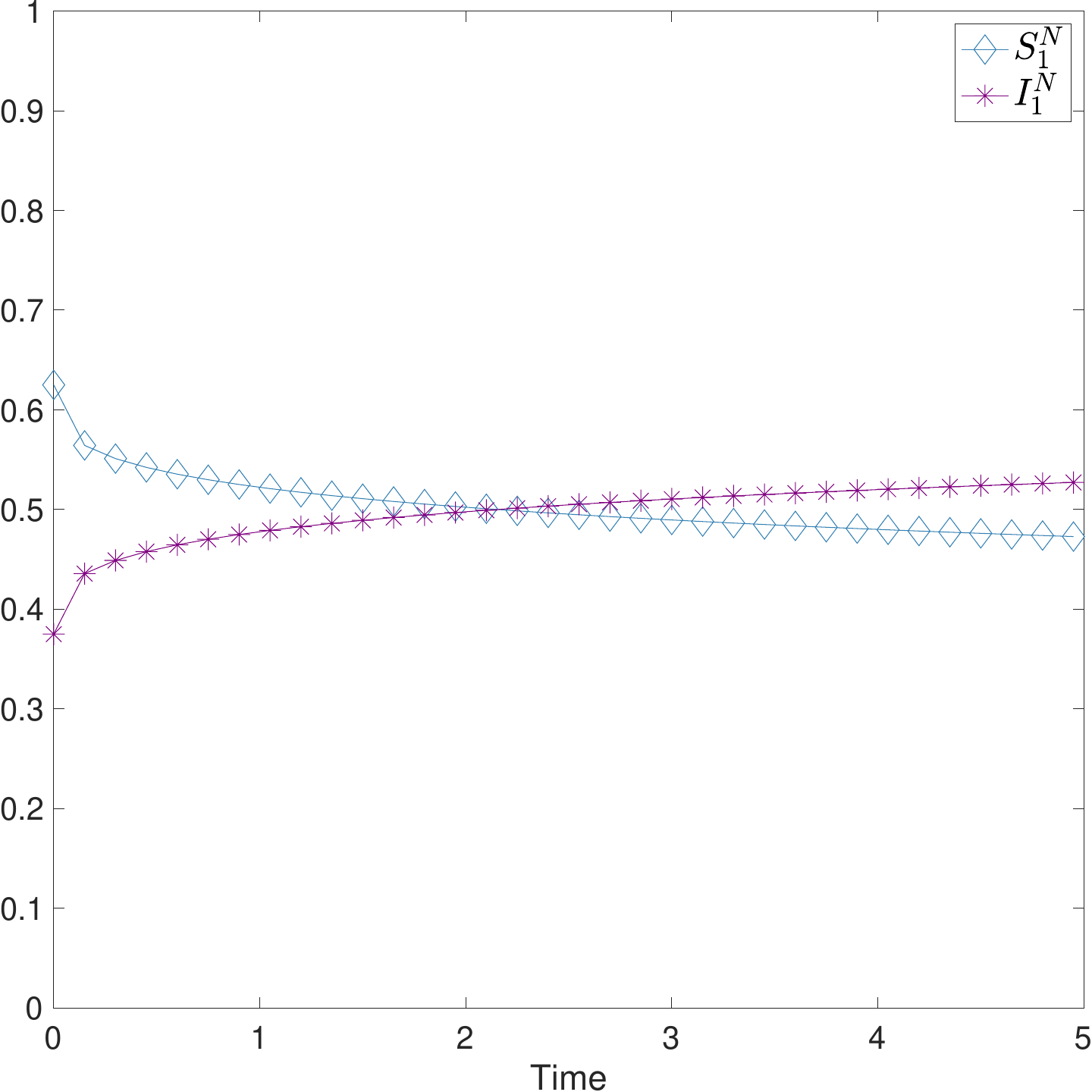}
}\,
\subfloat[][Numerical solutions  \eqref{eq:giga1}-\eqref{eq:giga2}.]{
\includegraphics[width=0.31\columnwidth]{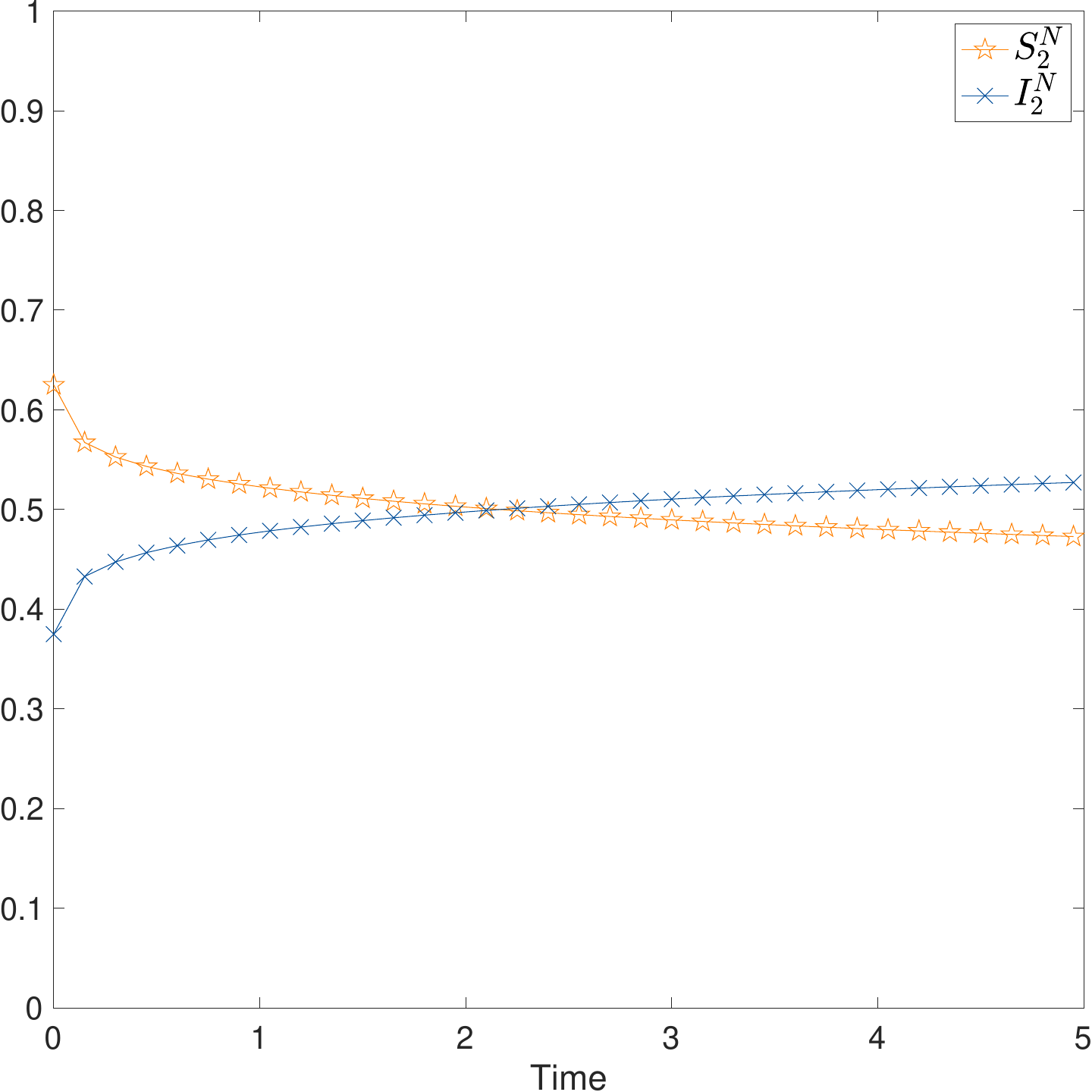}}
\caption{Comparison between the explicit and numerical fractional solutions to \eqref{eq:sisCap} with $\alpha=0.3$.}
\label{fig:alpha03}
\end{figure}

To further investigate on the three methodologies, we compute the $L^{\infty}$-norm of the difference between the exact fractional solutions \eqref{eq:solIfle}-\eqref{eq:solSfle} and the two numerical solutions \eqref{eq:saka1}-\eqref{eq:saka2} and \eqref{eq:giga1}-\eqref{eq:giga2} and between the two numerical solutions each others, as shown in Table \ref{tab:errore}. We observe that the errors range from orders of $10^{-5}$ to $10^{-3}$, increasing with respect to the decrease of $\alpha$. This fact further certifies the similarity between the three proposed methodologies.

\begin{table}[h!]
\renewcommand{\arraystretch}{1.3}
\centering
\begin{tabular}{|c|c|c|c|}
\hline
$\alpha$ & $\norm{\I^{F}-\I^{N}_{1}}_{\infty}$ & $\norm{\I^{F}-\I^{N}_{2}}_{\infty}$ & $\norm{\I^{N}_{1}-\I^{N}_{2}}_{\infty}$ \\\hline
0.99 & 1e--05 & 9e--04 & 9e--04\\\hline
0.7 & 1e--05 & 2e--03 & 2e--04\\\hline
0.3 & 3e--05 & 8e--03 & 8e--03\\\hline
\end{tabular}
\caption{Comparison of the $L^{\infty}$-norm between the solutions computed with the three methodologies for different values of $\alpha$.}
\label{tab:errore}
\end{table}

%\begin{equation}
%	f(\alpha) = t_{1}+\frac{t_{1}^{-\alpha}}{\alpha\Gamma(1-\alpha)}-\alpha(1-\alpha)
%\end{equation}

\subsubsection{Test with $c=0$} 
We focus now on the case of carrying capacity $c=0$. We fix this set of parameters: $\beta = 0.7$, $\gamma = 0.07$, $\mu=0.63$, $\sigma=1$ and $c = 0$. Moreover, the initial data are $\I(0)=1/(2\beta)$ and $\s(0)=1-\I(0)$, the final time is $T=1$ and the time step $\dt = 0.01$.

In Figure \ref{fig:alpha1_2} we compare the exact fractional solutions \eqref{eq:solIfle2}-\eqref{eq:solSfle2} and the two numerical solutions \eqref{eq:saka1}-\eqref{eq:saka2} and \eqref{eq:giga1}-\eqref{eq:giga2} for $\alpha=0.99$. Again, we observe that the fractional solutions, both explicit and numerical, perfectly overlap the solution to the $\sis$ model. In Figure \ref{fig:alpha07_2} we show the results obtained with $\alpha=0.7$. Analogously to the example with $c\neq0$, the point of intersection between the two densities of population slightly moves to the right with respect to the solution shown in Figure \ref{fig:alpha1_2}. Moreover, the three different methodologies produce again almost identical results. Finally, in Figure \ref{fig:alpha05_2} we show the results obtained with $\alpha=0.5$. In this case, the explicit fractional solutions \eqref{eq:solIfle2}-\eqref{eq:solSfle2} blow up in finite time, since the final time $T$ is greater than the radius of convergence, while the two numerical solutions show that the intersection point between the two curves further moves to the right with respect to Figure \ref{fig:alpha07_2}.

\begin{figure}[h!]
\centering
\subfloat[][Fractional solution \eqref{eq:solIfle2}-\eqref{eq:solSfle2}.]{
\includegraphics[width=0.31\columnwidth]{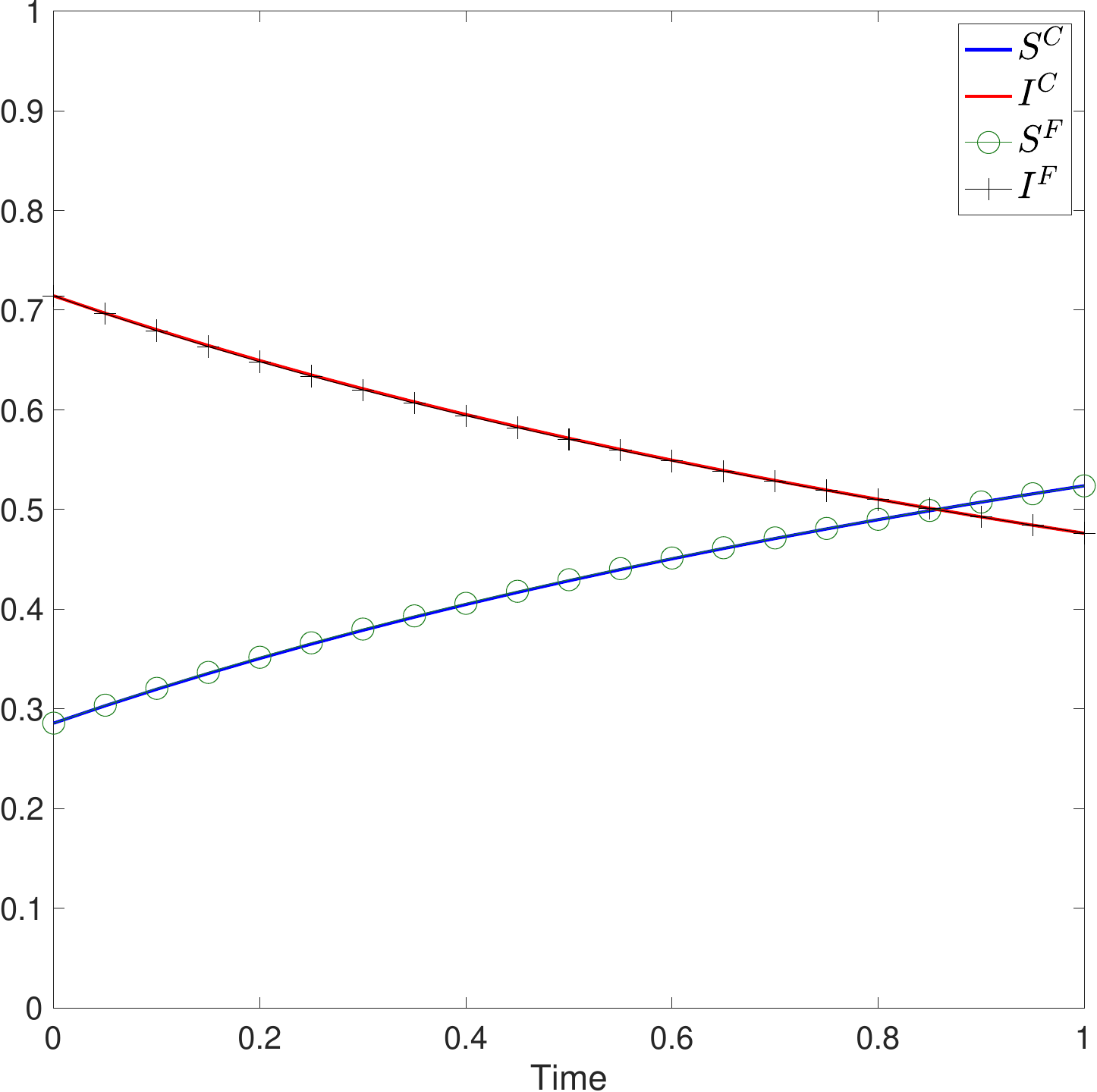}
}\,
\subfloat[][Numerical solution to \eqref{eq:sisCap} described as Method 1.]{
\includegraphics[width=0.31\columnwidth]{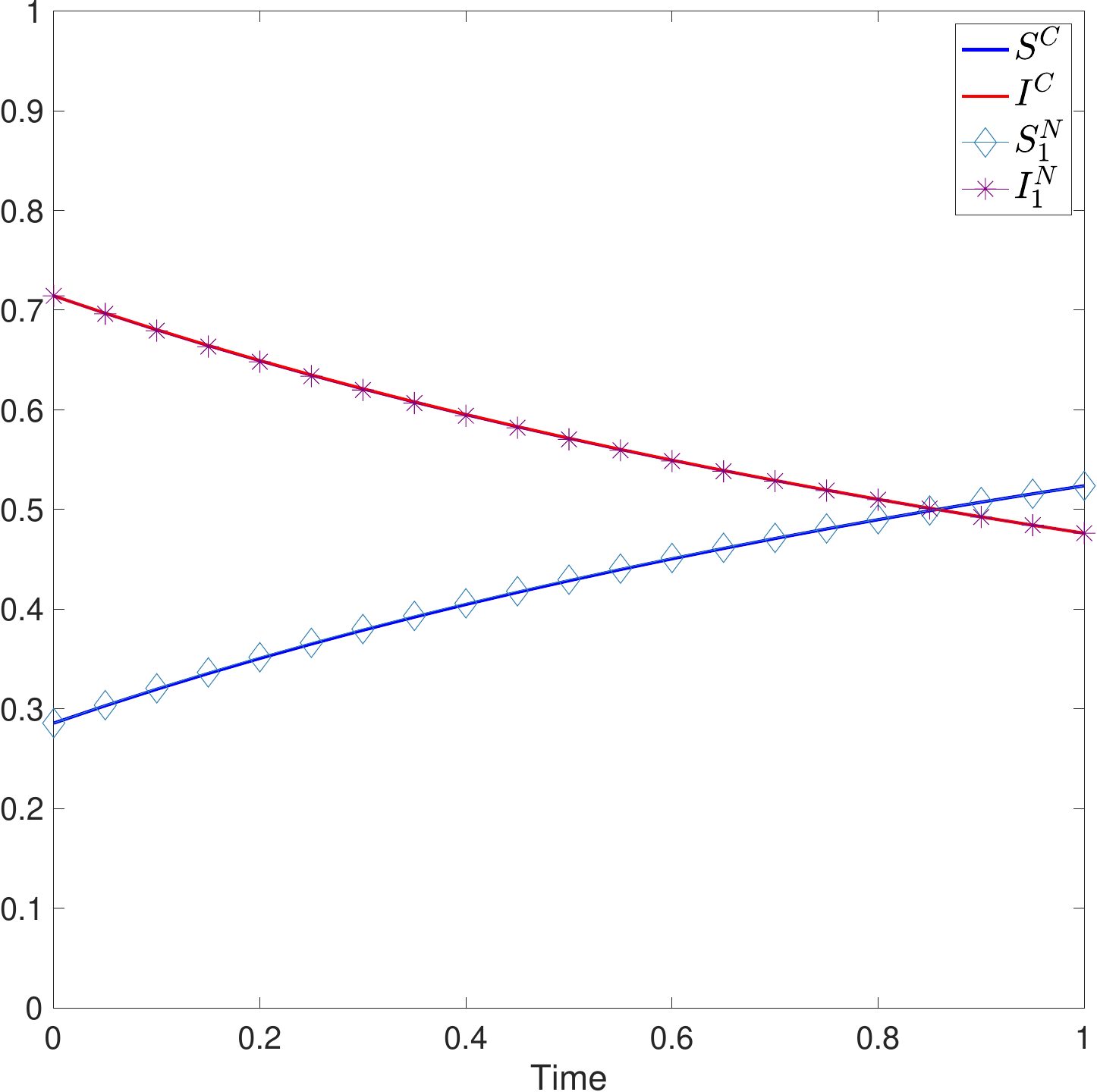}
}\,
\subfloat[][Numerical solution to \eqref{eq:sisCap} described as Method 2.]{
\includegraphics[width=0.31\columnwidth]{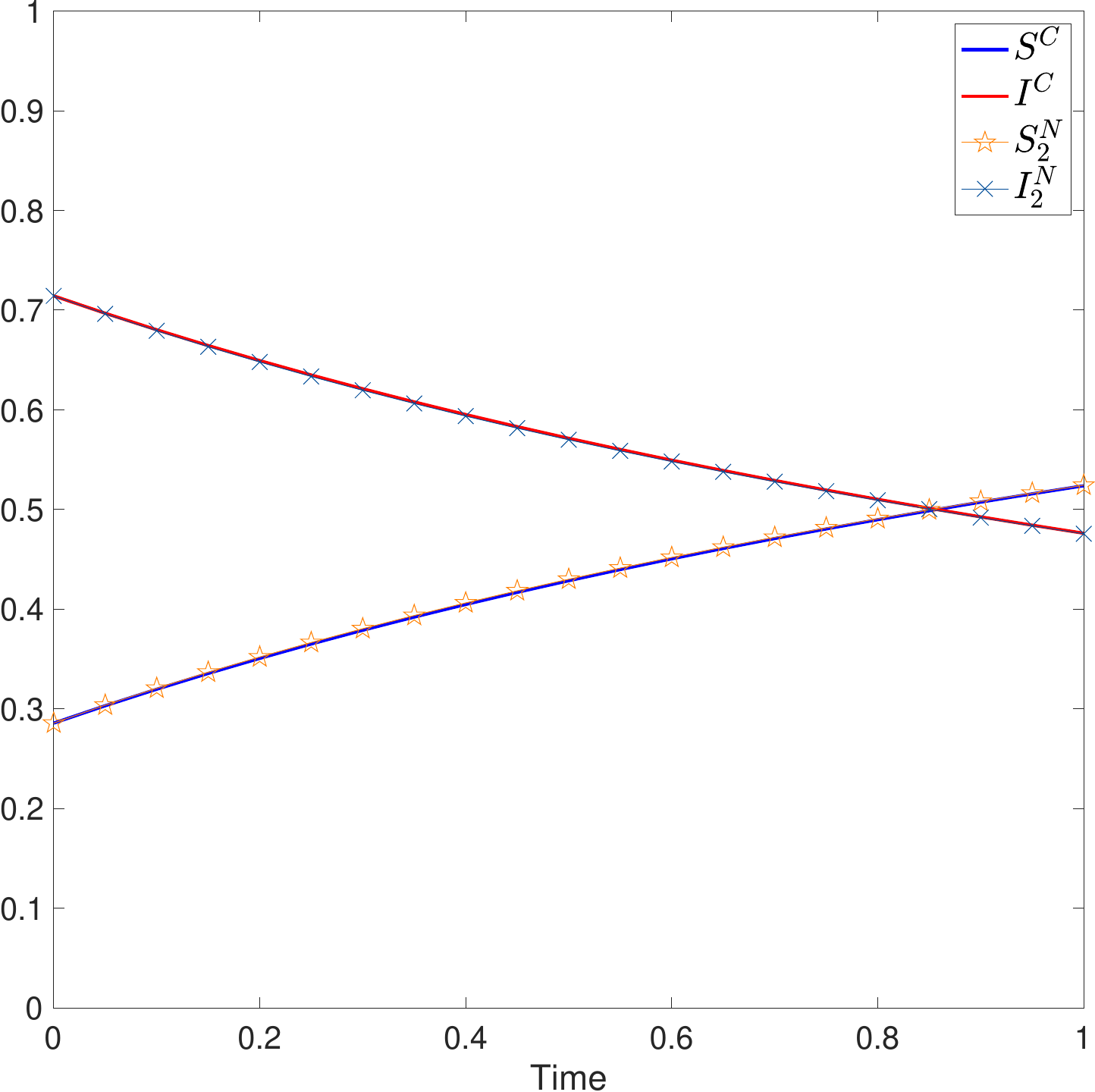}}
\caption{Comparison between the solutions to the SIS model and the fractional solutions to \eqref{eq:sisCap} with $\alpha=0.99$ (continuity w.r. to $\alpha$). }
\label{fig:alpha1_2}
\end{figure}

\begin{figure}[h!]
\centering
\subfloat[][Fractional solution to \eqref{eq:sisCap} \eqref{eq:solIfle2}-\eqref{eq:solSfle2}.]{
\includegraphics[width=0.31\columnwidth]{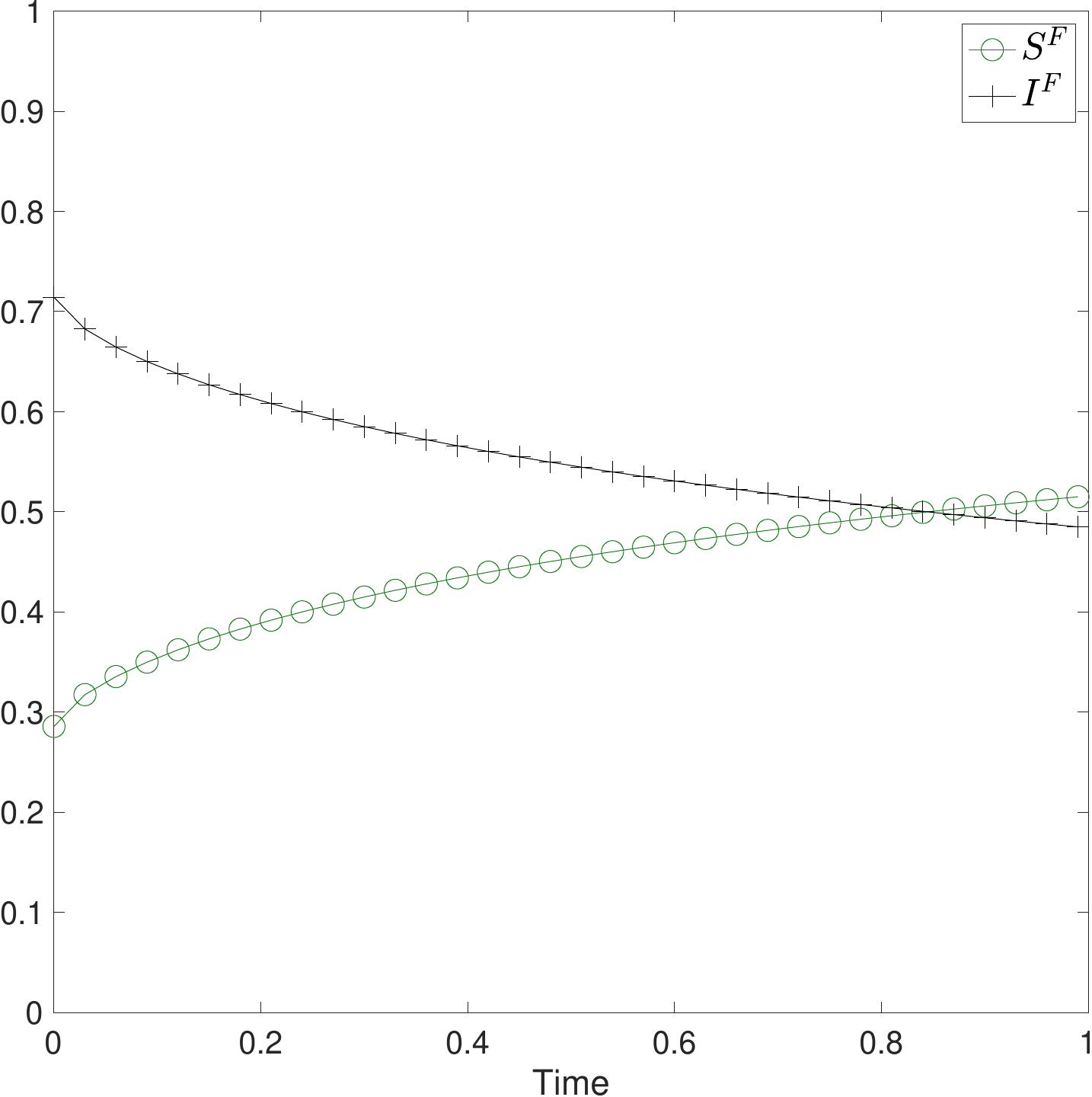}
}\,
\subfloat[][Numerical solution to \eqref{eq:sisCap} described as Method 1.]{
\includegraphics[width=0.31\columnwidth]{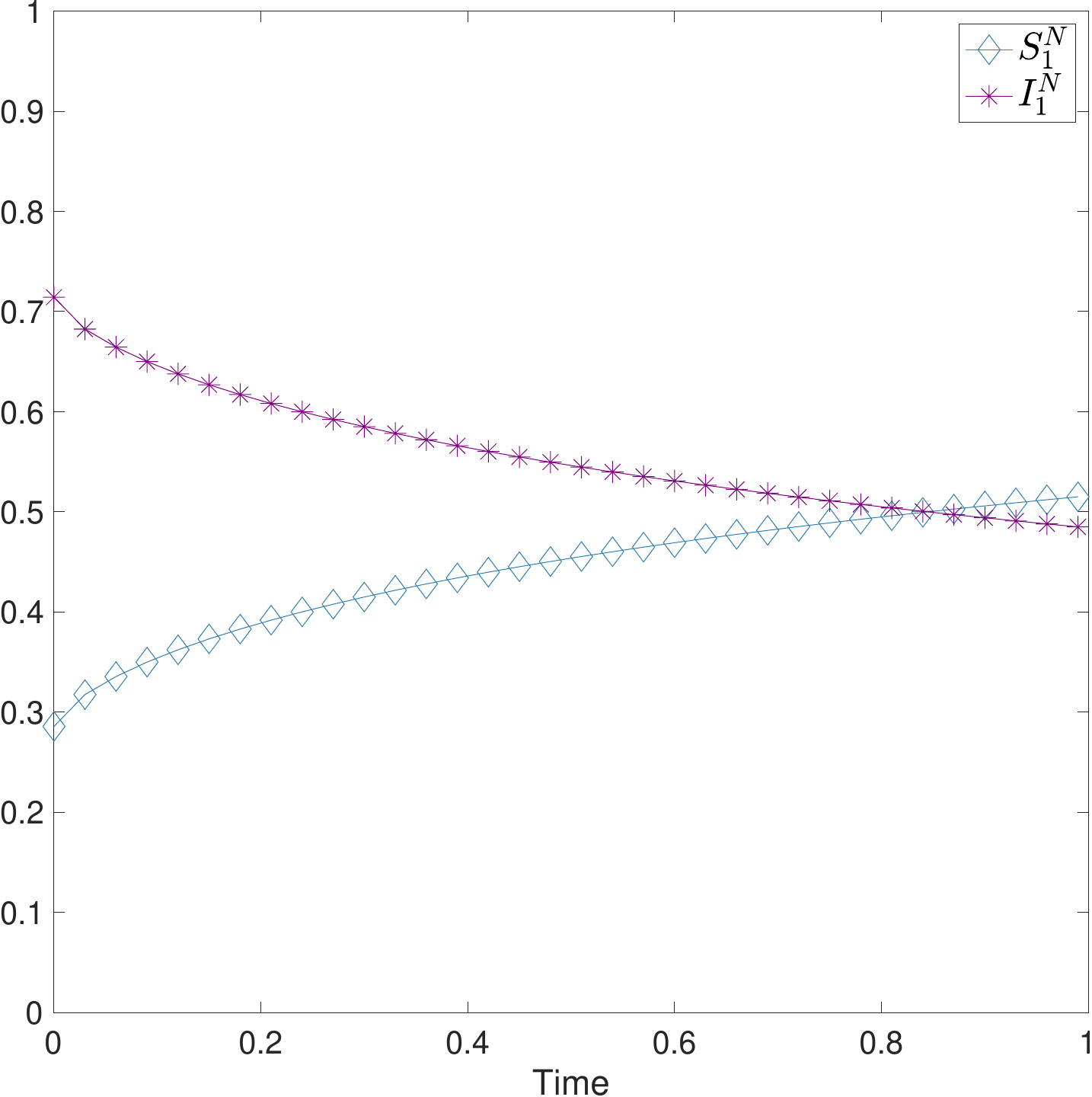}
}\,
\subfloat[][Numerical solution to \eqref{eq:sisCap} described as Method 2.]{
\includegraphics[width=0.31\columnwidth]{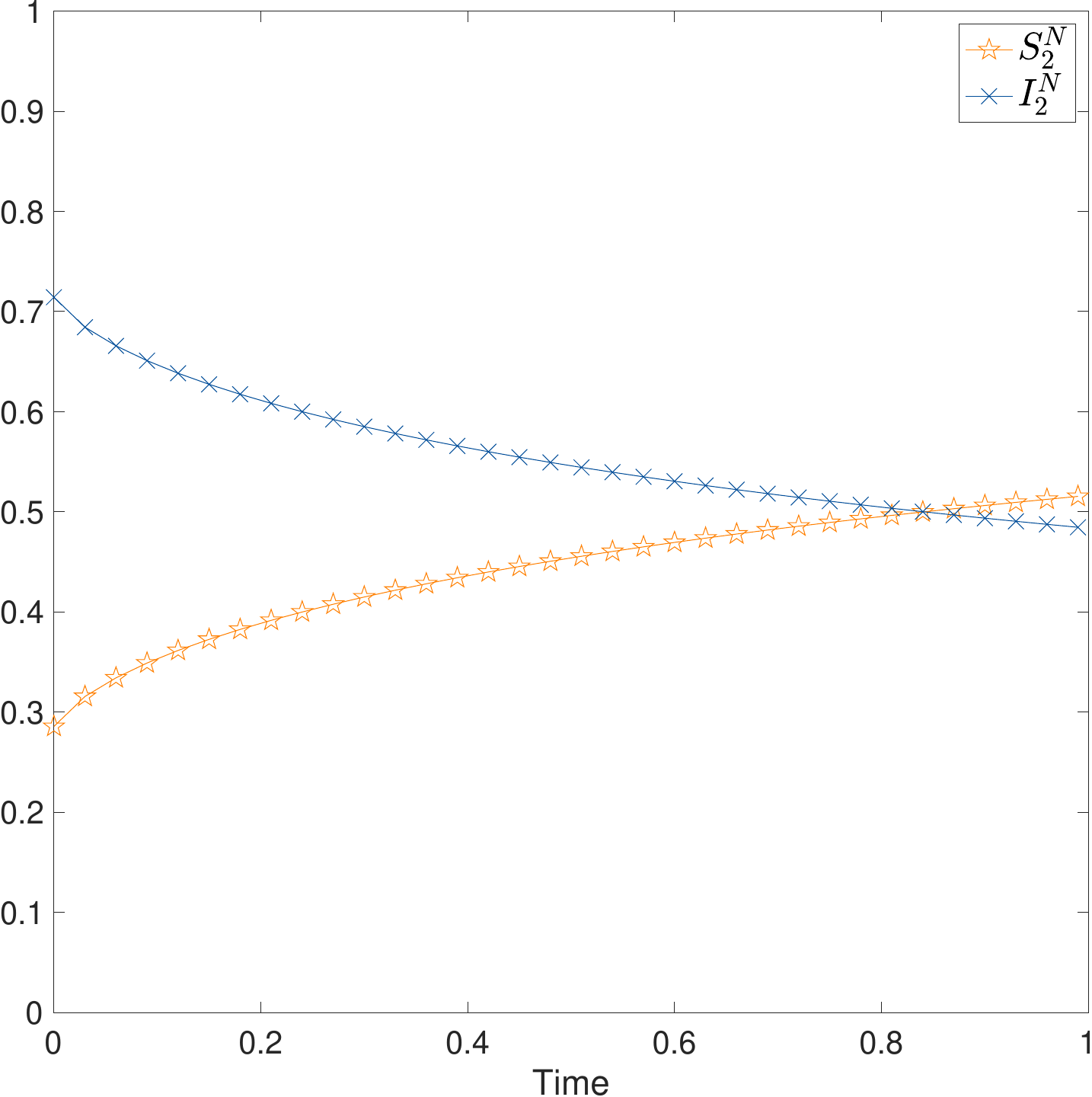}}
\caption{Comparison between the fractional solutions to \eqref{eq:sisCap} with $\alpha=0.7$. }
\label{fig:alpha07_2}
\end{figure}

\begin{figure}[h!]
\centering
\subfloat[][Fractional solution to \eqref{eq:sisCap} \eqref{eq:solIfle2}-\eqref{eq:solSfle2}.]{
\includegraphics[width=0.31\columnwidth]{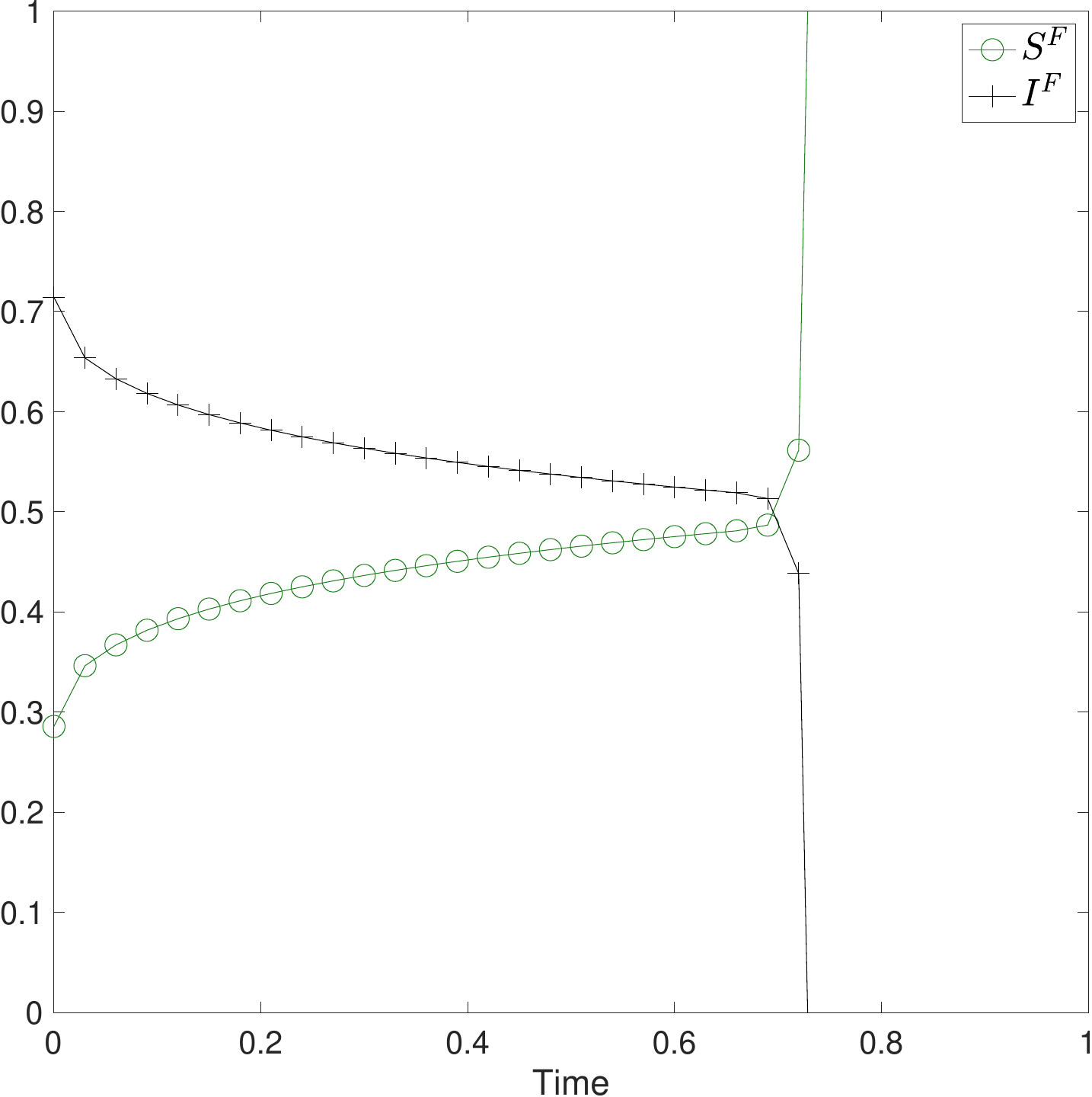}
}\,
\subfloat[][Numerical solution to \eqref{eq:sisCap} described as Method 1.]{
\includegraphics[width=0.31\columnwidth]{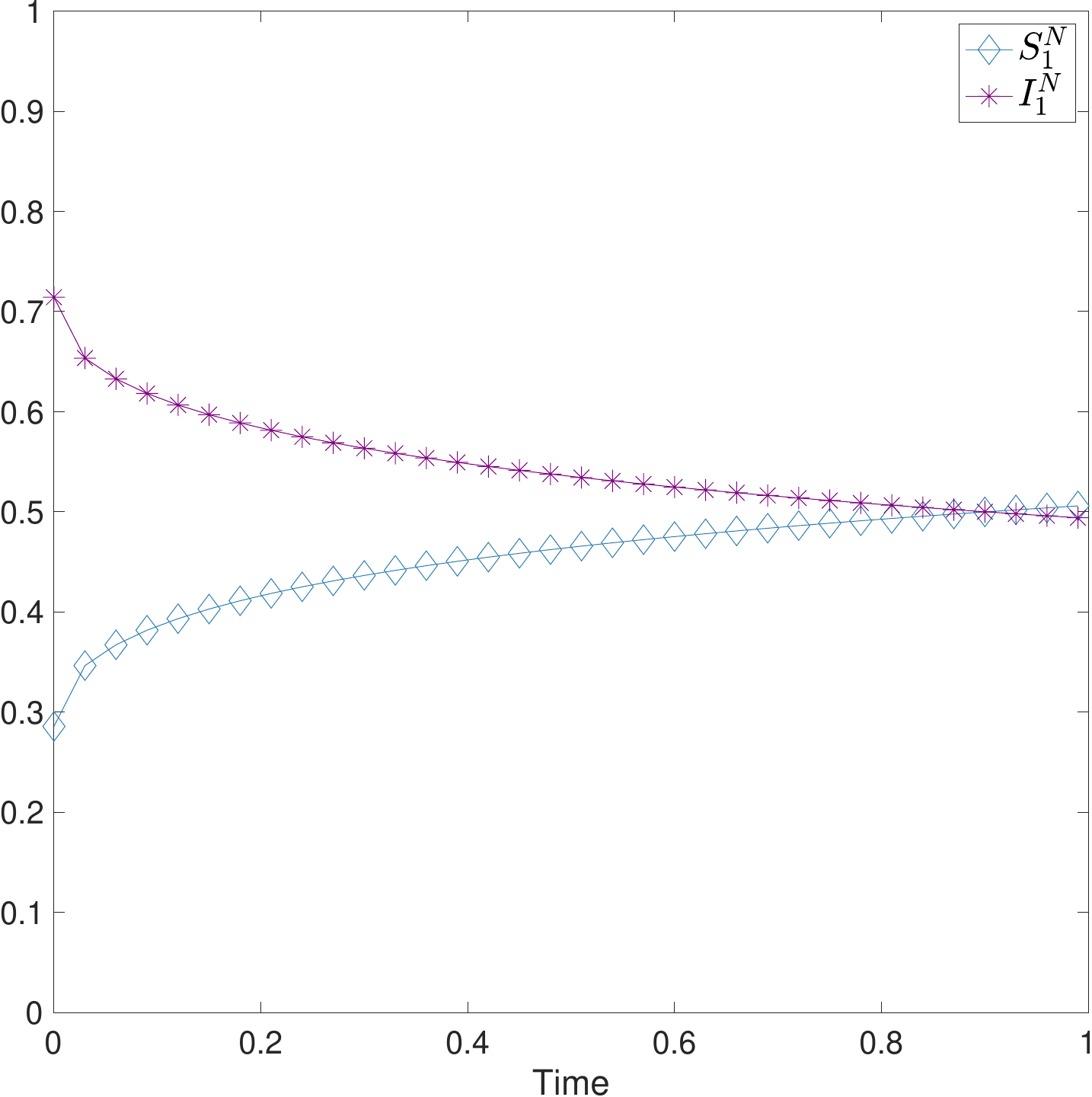}
}\,
\subfloat[][Numerical solution to \eqref{eq:sisCap} described as Method 2.]{
\includegraphics[width=0.31\columnwidth]{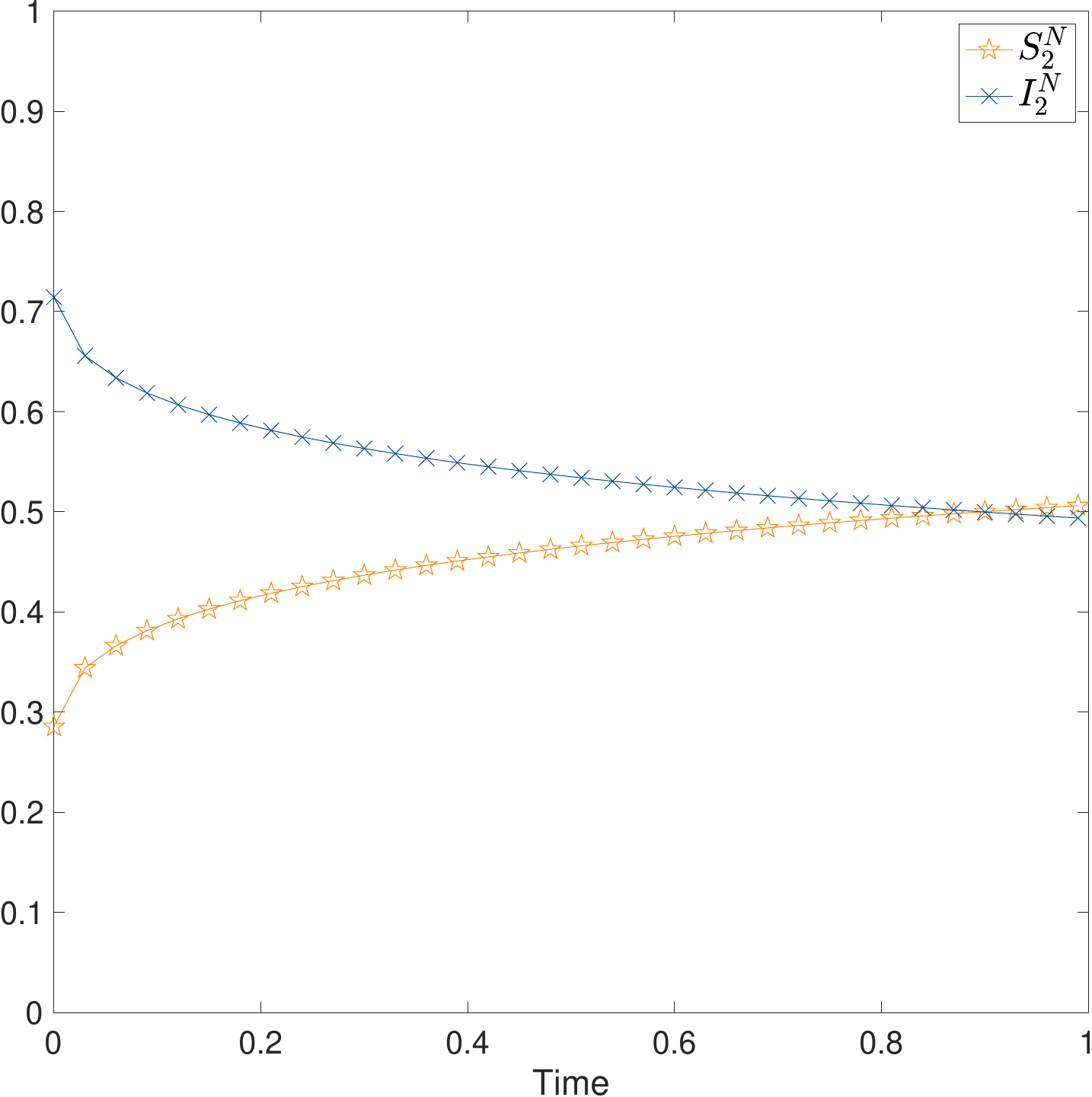}}
\caption{Comparison between the fractional solutions to \eqref{eq:sisCap} with $\alpha=0.5$. }
\label{fig:alpha05_2}
\end{figure}

\section{Conclusions}\label{sec:conc}
In this work we have studied the fractional $\sis$ model with constant population size. We have proposed an explicit representation of the solution to the fractional model under particular assumptions on parameters and initial data. By considering the basic reproduction number we rearrange the $\sis$ model and obtain a logistic equation. In the new formulation of the problem the carrying capacity has a new meaning based on the parameters of the $\sis$ model. We exploit such a formulation in order to study the fractional $\sis$ model and obtain a fruitful characterization of the problem, despite of many difficulties introduced by non-locality. In our formulation the carrying capacity can equal zero and this brings our attention to a different non-linear problem which in turns, it is related to the underlined $\sis$ model. We have introduced two different numerical schemes to approximate the model and perform numerical simulations, with which we have tested the proposed explicit solution. 

%Future investigations will concern the study of the $\sis$ model by considering the following aspects:
%\begin{itemize}[label=$\textendash$]
%\item the operator $\mathfrak{D}^\Phi_t$ with a Bernstein symbol $\Phi$ which well agrees with some underlying properties of the model;
%\item  non constant population size. In the literature, some models have been investigated in this direction and the population size is well-described in terms of the Mittag-Leffler function in the fractional framework. Based on the formulation above, our extension of the $\sis$ model will be concerned with the study of the non-homogeneous (with a forcing term) fractional logistic equation. As far as we know, this is an open problem.   
%\end{itemize}

\bibliographystyle{siam}
\bibliography{articoli}

\end{document}